\documentclass[a4paper,11pt]{article}
\usepackage{amsmath}                    
\usepackage{amssymb}                    
\usepackage{amsthm}                     
\usepackage{amsfonts}                   
\usepackage{subfigure}
\usepackage{color}
\usepackage{graphics,graphicx}
\usepackage{enumerate}
\usepackage{epsfig}
\usepackage{GGGraphics}

\usepackage{multirow}

\usepackage{dsfont}

\allowdisplaybreaks

\newtheorem{theorem}{Theorem}[section]

\newtheorem{proposition}[theorem]{Proposition}
\newtheorem{lemma}[theorem]{Lemma}
\newtheorem{remark}[theorem]{Remark}

\newtheorem{ass}[theorem]{Assumption}

\def\ts{\thinspace}

\renewcommand{\epsilon}{\varepsilon}
\newcommand{\eps}{\varepsilon}

\newcommand{\R}{\mathbb{R}}
\newcommand{\bigo}{\mathcal{O}}
\newcommand{\CC}{\mathcal{C}}

\def \IR{\mathbb R}
\def \IN{\mathbb N}

\def \IE{\mathbb E}

\def \eps{\epsilon}

\def \H{\mathcal H}
\def \s{\rm s}

\def \A{\mathcal{A}}
\def \B{\mathcal{B}}
\def \C{\mathcal{C}}
\def \D{\mathcal{D}}
\def \R{\mathcal{R}}

\DeclareMathOperator*{\limas}{lim\,a.s.}
\DeclareMathOperator*{\reg}{reg}

\title{
High-order integrator for sampling the invariant distribution of a class of parabolic SPDEs with additive space-time noise.
}

\author{ 
Charles-Edouard Br\'ehier\textsuperscript{1} and Gilles Vilmart\textsuperscript{2}
}

\begin{document}
\footnotetext[1]{
Univ Lyon, Universit\'e Claude Bernard Lyon 1, CNRS UMR 5208, Institut Camille Jordan, 43 blvd. du 11 novembre 1918, F-69622 Villeurbanne cedex, France. brehier@math.univ-lyon1.fr
}
\footnotetext[2]{
Universit\'e de Gen\`eve, Section de math\'ematiques, 2-4 rue du Li\`evre, CP 64, CH-1211 Gen\`eve 4, Switzerland, Gilles.Vilmart@unige.ch}

\maketitle

\begin{abstract}

We introduce a time-integrator to sample with high order of accuracy
the invariant distribution for a class of semilinear SPDEs driven by an additive space-time noise. 
Combined with a postprocessor, the new method is a modification with negligible overhead of the standard linearized implicit Euler-Maruyama method.
We first provide an analysis of the integrator when applied for SDEs (finite dimension), where we prove that the method has order $2$ for the approximation of the invariant distribution, instead of $1$.
We then perform a stability analysis of the integrator in the semilinear SPDE context, and we prove in a linear case that a higher order of convergence is achieved.
Numerical experiments, including the semilinear heat equation driven by space-time white noise, confirm
the theoretical findings and illustrate the  
efficiency of the approach.

\smallskip

 \end{abstract}


\section{Introduction}

We introduce an efficient integrator for the sampling of the invariant probability distribution of
a class of semilinear parabolic SPDEs with additive noise written as an abstract stochastic evolution equation (in the sense of \cite{DPZ14})
\begin{equation}\label{eq:SPDE}
du(t)   = \left(A u(t)+ F(u(t))\right)dt + d W^Q(t),\quad u(0)=u_0.
\end{equation}
Its solution $u(t)$ takes values in a separable infinite dimensional Hilbert space $\H$, with initial condition $u_0$ (assumed deterministic for simplicity). 
We assume that
$-A$ is a positive unbounded self-adjoint linear operator with an associated
sequence of positive eigenvalues $0<\lambda_1 \leq \lambda_2 \leq \ldots$ and an associated complete orthonormal family of eigenvectors $e_1,e_2,\ldots$
The coefficient $F: \H\rightarrow \H$ is a Lipschitz continuous nonlinearity, and it is assumed to derive from a continuously differentiable potential function $V:\H \rightarrow \IR$, {\it i.e.} $F=-D V$. Finally, we assume that $\bigl(W^Q(t)\bigr)_{t\in\IR^+}$ is a $Q$-Wiener process on $\H$ defined on a probability space $(\Omega,\mathcal{F},\mathbb{P})$ fulfilling the usual conditions; the covariance operator $Q:\H\rightarrow \H$ is a bounded, non-negative self-adjoint linear operator, such that $Qe_i=q_ie_i$ for some bounded sequence of real numbers $(q_i)_{i\in\IN^*}$, 
where we use the notation $\IN^*=\{1,2,3,\ldots\}$. 
We assume the following trace condition:
\begin{equation}\label{eq:conditionTrace}
\overline{\s}=\sup\left\{ \s \in(0,1)~:~{\rm Trace}\Bigl( (-A)^{-1+\s}Q\Bigr)<+\infty\right\}>0.
\end{equation}
Then, there exists a unique mild solution of \eqref{eq:SPDE} on $\IR^+$ (see \cite[Chapter $7$]{DPZ14}), {\it i.e.} an $\H$-valued continuous stochastic process $\bigl(u(t)\bigr)_{t\in\IR^+}$ which satisfies
\begin{equation}\label{eq:mild_solution}
u(t)=e^{tA}u_0+\int_{0}^{t}e^{(t-s)A}F(u(s))ds+ \int_{0}^{t}e^{(t-s)A}dW^Q(s).
\end{equation}
This abstract setting includes the stochastic semilinear heat equation 
where $A=\Delta = \sum_{j=1}^d\frac{\partial^2 }{\partial x_j^2}$ 
in $\H=L^2(\D)$ for some open smooth bounded domain $\D\subset \IR^d$,\looseness-3 
\begin{equation} \label{eq:heat1d}
\frac{\partial u}{\partial t}(x,t) = \Delta u(x,t) + f(u(x,t)) + \frac{\partial W^Q}{\partial t}(x,t),
\end{equation}
with homogeneous Dirichlet boundary conditions on $\partial \D$. In~\eqref{eq:heat1d}, the coefficient $f:\IR\to\IR$ is a smooth, Lipschitz function; the associated {\it Nemytskii} coefficient $F:u\in \H\mapsto f\circ u\in\H$ satisfies $F(u)=-DV(u)$ for all $u\in \H$, with $V(u)=-\int_{0}^{1}\langle F(\theta u),u\rangle_{\H}d\theta$. 

If $d=1$, one can consider space-time white noise in~\eqref{eq:heat1d}, ({\it i.e.} with the identity covariance operator $Q=I$; this choice yields $\overline s=1/2$ in \eqref{eq:conditionTrace}). When $d>1$, a nontrivial covariance operator $Q\neq I$ is required, yielding noise which is white in time and colored in space. Notice also that \eqref{eq:conditionTrace} is automatically satisfied with $\overline s=1$ for a trace-class noise with ${\rm Trace}(Q)<\infty$.

To study the long-time behavior of the process $u$, we make the additional assumption that $F$ in \eqref{eq:SPDE} is Lipschitz continuous with constant ${\rm L}>0$ such that ${\rm L} < \lambda_1=\min_{p\in\IN^*}\lambda_p$.
Then (see e.g. Section $8.6$ in \cite{DPZ96}) Eq. \eqref{eq:SPDE}
admits a unique invariant distribution $\mu_\infty$. 
This means that for all (smooth and Lipschitz) test functions $\phi:\H\rightarrow \IR$, and for all initial conditions $u_0$,
$$
\limas_{T \rightarrow \infty} \frac{1}{T}\int_{0}^T \phi(u(t))dt =\int_{\H}\phi(y)d\mu_\infty(y),
$$
where the notation $\limas$ means that the limit holds with probability $1$. Moreover (see e.g. Section $6.3$ in \cite{DPZ96}) $u(t)$ converges in law to $\mu_\infty$ exponentially fast in the following sense:
for all $t>0$, and any test function $\phi$,
\begin{equation} \label{eq:expergoSPDE}
\left| \IE\phi(u(t)) - \int_{\H}\phi(y)d\mu_\infty(y) \right| \leq C(\phi,u_0) e^{-(\lambda_1-L) t}
\end{equation}
where $C(\phi,u_0)$ is independent of $t$.
It is a standard approach to take advantage of the estimate \eqref{eq:expergoSPDE} to compute
ergodic integrals of the form $\int_{\H}\phi(y)d\mu_\infty(y)$. To do so, in practice one needs to rely on a discretization of the evolution equation. 
We thus now present two implicit-explicit time-discretization schemes.


\paragraph{Linearized implicit Euler method}
We first consider the simplest numerical scheme, which is referred to as the linearized implicit Euler method in this article. Given a constant timestep size $h>0$, it is defined by $
v_{n+1}=v_{n}+hAv_{n+1}+hF(v_{n})+\sqrt{h}\xi_{n}^{Q},
$
equivalently
\begin{equation}\label{eq:v_n}
v_{n+1}=J_1\left(v_{n}+hF(v_{n})+\sqrt{h}\xi_{n}^{Q}\right),
\end{equation}
where $v_0=u(0)=u_0$, $J_1=(I-hA)^{-1}$ and $\xi_{n}^{Q}=h^{-1/2}\bigl(W^Q((n+1)h)-W^Q(nh)\bigr)$. 
For a fixed final time $T>0$, it is known that the scheme \eqref{eq:v_n} applied to \eqref{eq:SPDE} has weak order of accuracy $q$ for all $q<\overline s$, {\it i.e.} it satisfies for all $h$ small enough,
\begin{equation} \label{eq:weakorder}
\left|\IE(\phi(v_n)) -  \IE(\phi(u(t_n)))\right| \leq C(u_0,\phi,T) h^q,
\end{equation}
for all $t_n=nh \leq T$, where $C(u_0,\phi,T)$ is independent of $n,h$. 

%
The weak order of accuracy \eqref{eq:weakorder} has been analyzed in \cite{DeP09} for \eqref{eq:v_n} applied to the heat equation in the case $F=0$ and in \cite{D11} in the semilinear case; see also \cite{WG13} for the case of colored noise in dimension $d>1$ and \cite{JK15} where different techniques are used. 
Concerning the approximation of the invariant distribution of the heat equation, it is proved in \cite{B14} that for $d=1$ and $Q=I$, one has for all time $t_n$, similarly to \eqref{eq:expergoSPDE}, the exponential convergence property
\begin{equation} \label{eq:difference1iexp}
\left|\IE(\phi(v_n)) - \int_{\H} \phi(y) d\mu_\infty(y)\right| \leq K(u_0,\phi)e^{-\lambda t_n} + C(\phi) h^r
\end{equation} 
for all $r<\overline s$ with $\overline{\s}=1/2$, where $C(\phi),K(u_0,\phi)$ are independent of $n$ and $h$, and $\lambda$ is independent of $\phi,h,n$. The proof is based on the analysis of the weak approximation using the tools of \cite{D11} (expansion of the error using the backward Kolmogorov equation, Malliavin calculus), with a proof that constants $C(u_0,\phi,T)$ in \eqref{eq:weakorder} can be chosen independent of the final time $T$, in the spirit of \cite{Tal84}.

\paragraph{New method}
The contribution of this paper is the introduction of a modification, denoted $u_n$, 
of the standard Euler scheme \eqref{eq:v_n}, together with a postprocessor, denoted $\overline u_n$, which permits to achieve the estimate \eqref{eq:difference1iexp} with the higher order $r<\overline s+1$ instead of 
$r<\overline s$, when in~\eqref{eq:difference1iexp} $v_n$ is replaced with $\overline u_n$. Precisely, the postprocessed method is given by two sequences $\bigl(u_n\bigr)_{n\geq 0}$ and $\bigl(\overline{u}_n\bigr)_{n\geq 0}$ in $\H$ defined by
\begin{eqnarray}
 u_{n+1}&=& 
J_1 \Big(u_{n}+hF\big(u_n+\frac{1}{2}\sqrt{h}J_2\xi_{n}^{Q}\big)
+\frac{\sqrt{2}-1}{2}  \sqrt h J_2\xi_{n}^{Q} \Big)
+\frac{3-\sqrt{2}}{2}\sqrt{h}J_2\xi_{n}^{Q},
\label{eq:u_n}\\
\overline{u}_{n}&=& u_n+\frac{1}{2} \sqrt{h}J_3\xi_{n}^Q, \label{eq:u_n_bar}
\end{eqnarray}
where we introduce additional operators $J_2,J_3$ as follows: 
$$J_2=(I-\frac {3-\sqrt2 }2  hA)^{-1}, \quad J_3QJ_3^T = (I-\frac h2 A)^{-1}Q.$$ 
Note that $J_3$ is not determined uniquely by the equation above: one may define $J_3 = (I-\frac h2 A)^{-1/2}$, or use a Cholesky decomposition.
We emphasize the importance of the postprocessing operation~\eqref{eq:u_n_bar}: indeed, the order of convergence to the invariant distribution $\mu_{\infty}$ is not improved if in~\eqref{eq:difference1iexp} $v_n$ is replaced with $u_n$, instead of $\overline u_n$. 
Notice that the new scheme \eqref{eq:u_n}-\eqref{eq:u_n_bar} has a negligible overcost compared to  \eqref{eq:v_n}. First, the postprocessor \eqref{eq:u_n_bar} needs not to be computed at each time step, but it can be computed only once at the end of each numerical trajectory. Second, computing $J_2\xi_n^Q$ at each timestep is the only additional calculation. Third, we prove that the constant $\lambda>0$ in~\eqref{eq:difference1iexp} can be chosen with the same size for both methods: thus in terms of cost (number of timesteps required to achieve a given accuracy on the left-hand side of~\eqref{eq:difference1iexp}), our new integrator performs better than the standard 
scheme \eqref{eq:v_n}.

The fact that the long-time weak order of accuracy $r$ in \eqref{eq:difference1iexp} can be made strictly larger than the short time accuracy $q$ in \eqref{eq:weakorder} is not surprising:
this is known for SDEs ({\it i.e.} in finite dimension in the terminology of this paper), see \cite{AEL08, LM13,AVZ14a,Vil15} in the context of Brownian dynamics and \cite{BRO10,AVZ15b} in the context of Langevin dynamics, where integrators with low weak order, typically $q=1$ in \eqref{eq:weakorder}, are shown to achieve a high order $r>q$ over long times \eqref{eq:difference1iexp} for sampling the invariant measure.
Inspired by these recent advances, the popular technique of processing for deterministic differential equations \cite{butcher69teo}
was recently extended to the stochastic context in \cite{Vil15}, and serves as a crucial ingredient to derive the new method proposed in this paper. In our context, this idea is to enhance the accuracy of the modified numerical method $u_n$ in \eqref{eq:u_n}
by applying a suitable change of variables $u_n\mapsto \overline u_n$ defined in \eqref{eq:u_n_bar}.  

Alternatively, note that high order integrators in the strong sense (approximation of the trajectories  instead of the distribution) for parabolic problems of the form \eqref{eq:SPDE} 
are proposed in \cite{JeK09,Jen11,JKW11}; however these schemes belong to the class of exponential integrators, while the method proposed in this paper avoids the computation of matrix exponentials. 
We also mention another natural integrator for \eqref{eq:SPDE}, which is the (stochastic) trapezoidal method (also known as the Crank-Nicolson method),
\begin{equation} \label{eq:CN}
u_{n+1} = u_n + \frac h2 A(u_{n} +u_{n+1}) + h F(u_n) + \sqrt h \xi_n^Q.
\end{equation}
However, in contrast to \eqref{eq:v_n},
the scheme \eqref{eq:CN} is not $L$-stable, a desirable property for severely stiff problems, already in the deterministic literature \cite{hairer93sod}. This makes
exponential convergence estimates of the form \eqref{eq:difference1iexp} not true in general for the scheme \eqref{eq:CN}.\looseness-3

Finally, since the space $\H$ is infinite dimensional, a space discretization scheme is required in practice, \emph{e.g.} finite differences or finite elements -- see Section~\ref{sec:num}. In this article, we only focus on the time-discretization issue.
We mention \cite{LoS07} where a postprocessing technique is applied to improve the spatial discretization in the strong sense.

\paragraph{Outline and main results}
This paper is organized as follows.
In section~\ref{sec:finitedim}, we explain the derivation of a generalized nonlinear version of the integrator~\eqref{eq:u_n}-\eqref{eq:u_n_bar} in the context of finite-dimensional SDEs, where many standard analysis tools are available compared to the SPDE context.
We prove that the new integrator has order~$2$ for the approximation of the invariant distribution of nonlinear ergodic SDEs, instead of order $1$ for the standard Euler scheme.
Section \ref{sec:SPDE} is devoted to the analysis of the integrator~\eqref{eq:u_n}-\eqref{eq:u_n_bar} for the SPDEs~\eqref{eq:SPDE}: 
we detail the abstract Hilbert space setting (Section~\ref{sec:hyp}), we show the stability and ergodicity properties of the integrator (Section~\ref{sec:timemeth}), we prove in a simplified linear case
the improved order of accuracy $\overline s +1$ of the new method for SPDEs, in contast to the order $\overline{\s}$ for the standard Euler scheme (Section~\ref{sec:analysis_order}), 
and we show that the proposed scheme exhibits the correct spatial regularity of the invariant distribution in terms of Sobolev-like spaces (Section~\ref{sec:qualitative}).
Finally, Section \ref{sec:num} is dedicated to numerical experiments which confirm the theoretical findings and illustrate the efficiency of the new method.

\section{New high order integrator: derivation and analysis in finite dimension}
\label{sec:finitedim}

In the context of SDEs in finite dimension $N\in\IN^*$, we consider the more general case of a nonlinear system of the form
\begin{equation} \label{eq:f1f2}
dX(t)= \big(f_1(X(t)) + f_2(X(t))\big) dt + \sigma dW ^Q(t),\qquad X(0)=X_0,
\end{equation}
with solution $X(t)$ in $\IR^N$. The nonlinearities $f_1,f_2:\mathbb{R}^N\rightarrow \mathbb{R}^N$ are smooth and Lipschitz functions such that $f_1(x) + f_2(x) = f_0(x)= -\nabla V_0(x)$ for some potential function $V_0:\IR^N\to \IR$, where $f_1$ is a  term to be treated implicitly and $f_2$ is a 
term to be treated explicitly. The initial condition $X_0$ is assumed deterministic for simplicity. We also define the $Q$-Wiener process $W^{Q}(t) =Q^{1/2}W(t)$ 
where $Q$ now denotes a $N\times N$ symmetric positive definite matrix
and $W(t)$ is a standard $N$-dimensional Wiener process, and $\sigma>0$ is a fixed constant. 

We introduce the following new implicit-explicit scheme for sampling with high order two the invariant measure of \eqref{eq:f1f2},
\begin{eqnarray}\label{eq:scheme_SDE}
X_{n+1} &=& X_n + h f_1\left(X_{n+1} + \frac{-2+\sqrt5}2  J_{n,2}\sigma \sqrt h \xi_n^Q\right) +  h  f_2\left(X_{n} + \frac 12 J_{n,2}\sigma \sqrt h \xi_n^Q\right)  \nonumber \\*
&+& \big(\frac{1-\sqrt2+\sqrt{5}}2 J_{n,1}^{-1} + \frac{1+\sqrt2-\sqrt{5}}2 \big)J_{n,2}\sigma \sqrt h \xi_n^Q,  \nonumber \\*
\overline X_{n} &=&  X_n + \frac{1}2 J_{n,3} \sigma\sqrt h \xi_{n}^Q,  \label{eq:newmeth} 
\end{eqnarray}
where $J_{n,1},J_{n,2},J_{n,3}$ are given by
$$
J_{n,1} = (I-hf_1'(X_n))^{-1},\ J_{n,2} = (I- \frac{3-\sqrt 2 }2  hf_1'(X_n))^{-1},\ J_{n,3}QJ_{n,3}^T = (I-\frac h2  f_1'(X_n))^{-1}.
$$

We emphasize that the matrix inverses $J_{n,1},J_{n,2},J_{n,3}$ are used only in the notations to define the scheme but should not be computed in practice.  Indeed, in practical implementations,
a LU decomposition should be used in place of computing matrix inverses. Moreover, in the semilinear case \eqref{eq:heatspdedisc}, this decomposition needs to be done only once and may be used for all further iterations.

\begin{remark} \label{rem:leimkhuler}
After a spatial discretization with finite differences or finite elements of the SPDE \eqref{eq:SPDE}, in general one arrives at a system of stiff SDEs in $\IR^N$ with large dimension $N$ of the form \eqref{eq:f1f2}, 
\begin{equation} \label{eq:heatspdedisc}
d X(t) = A X(t) dt + f(X(t)) dt + \sigma  dW^Q(t).
\end{equation}
where assume that $f(x) = -\nabla V(x)$ for some $V:\IR^N\to \IR$.
It corresponds to the special case
$f_1(x) = Ax
$
where $-A$ now denotes a $N\times N$ symmetric positive definite matrix.
In this case, the above scheme \eqref{eq:newmeth} simplifies to 
the integrator \eqref{eq:u_n}.
Moreover, in the special case $f_1=0$, the method \eqref{eq:newmeth} reduces to
\begin{equation} \label{eq:leimkhuler}
X_{n+1} = X_n + h  f\left(X_n + \frac12 \sigma \sqrt{h} \xi_n\right) 
 + \sigma \sqrt{h} \xi_n,
\qquad
\overline X_n=  X_n + \frac12 \sigma \sqrt{h} \xi_n,
\end{equation}
a method which was first proposed in \cite{LM13} and analyzed in \cite{LMT14}, 
although it was constructed in another manner using a non-Markovian formulation.
\end{remark}


For simplicity, we assume for the remaining of this section that $Q=I$ is the identity matrix, without loosing generality, applying the appropriate change of variable in (fixed) finite dimension. Section~\ref{sec:stab_SDE} is devoted to a stability analysis in the case of an Ornstein-Uhlenbeck process: we prove L-stability and exactness results, which play an important role in the efficiency of the integrator for semilinear SPDEs. Section~\ref{sec:nonlineardimfinie} contains details on the construction of the method in order to satisfy the above properties and to achieve order two of accuracy for sampling the invariant measure.

\subsection{Stability analysis}\label{sec:stab_SDE}


For the study of the stability of stochastic integrators applied to stiff SDEs with additive noise of the form \eqref{eq:f1f2},
a widely used test problem is the following scalar SDE problem (Ornstein-Uhlenbeck process)
\begin{equation} \label{eq:OU}
dX = - \lambda Xdt + \sigma dW(t),
\end{equation}
where $\lambda,\sigma>0$ are fixed constants.  
Notice that this test equation provides only a useful insight but no rigorous general conclusion on the numerical long-time behaviour for nonlinear problems with additive noise, see \cite{BRK11} and references therein. A rigorous analysis of the proposed scheme in our semilinear SPDE context is presented in Section 
\ref{sec:timemeth}.
Notice also that other test equations are used in the literature in the case of multiplicative noise, see \cite{SaM96,Hig00,SaM02,BBT04,Toc05,RaB08,BuC10}. 
Consider a one step method of the form
\begin{equation} \label{eq:XnAB}
X_{n+1} = \A(z)X_n + \B(z)\sqrt{h} \sigma \xi_n,\qquad z=-\lambda h,
\end{equation}
where $h$ is the stepsize, $\A(z),\B(z)$ are analytic functions, $\xi_n \sim \mathcal{N}(0,1)$ are independent Gaussian random variables.
The SDE \eqref{eq:OU} is ergodic with the unique invariant measure a Gaussian with mean zero and variance $\frac{\sigma^2}{2\lambda}$, {\it i.e} with density 
$\rho_\infty(x) = \sqrt{\frac{\lambda}{\pi\sigma^2}} \exp({-\frac{\lambda}{\sigma^2} x^2})$.
Indeed, for any initial condition $X_0=x$, the exact solution is a Gaussian random variable, with
$
\lim_{t\rightarrow \infty} \mathbb{E}(X(t)) = 0$ and $\lim_{t\rightarrow \infty} \mathbb{E}(|X(t)|^2) = \frac{\sigma^2}{2\lambda}.
$
The second-order moment $\mathbb{E}(|X_n|^2)$ remains bounded as $n\rightarrow \infty$ if $|\A(z)|<1$, which corresponds to the mean-square stability condition. 
The condition $|\A(z)|\leq 1$ for all $z$ with negative real part is called $A$-stability in the deterministic literature \cite{hairer93sod}. This is a desirable property of numerical integrators for stiff problems, because it permits to avoid a severe timestep size restriction. 
For $|\A(z)|<1$ (stability condition), we obtain 
\begin{equation} \label{eq:defR}
\lim_{n\rightarrow \infty} \mathbb{E}(|X_n|^2) = \frac{\sigma^2}{2\lambda} \R(z), \quad \mbox{ where } 
\R(z)=\frac{-2z\B(z)^2}{1-\A(z)^2}.
\end{equation}
We see that the method is exact (for the approximation of the invariant distribution) if and only if $\R(z)=1$ for all $z$; 
this is the case for instance for the trapezoidal method \eqref{eq:CN}, which is such that 
$
\A(z) = \frac{1+z/2}{1-z/2}, \ \B(z) = \frac{1}{1-z/2}.
$
However, in addition to $A$-stability, a desirable property of Runge-Kutta methods for very stiff problems is $L$-stability \cite{hairer93sod}, namely $\A(\infty)=0$; this is not satisfied by the trapezoidal method, for which $\A(\infty)=-1$. 
The following proposition states that for Runge-Kutta type methods, where $\A(z),\B(z)$ are rational functions, $L$-stability ({\it i.e.} $\A(\infty)=0$) is incompatible 
with the exactness for the invariant distribution, i.e. $\R(z)\equiv 1$. 
\begin{proposition}
Consider a method of the form \eqref{eq:XnAB} where $\A(z),\B(z)$ are rational functions.
If the method samples exactly the invariant distribution of \eqref{eq:OU} ({\em i.e.} $\R(z)\equiv 1$), then $|\A(\infty)|=1$. 
\end{proposition}
\begin{proof}
If $\A(z)$ is a rational function with $|\A(\infty)| \ne 1$,  there exist a constant $C\in\IR$ and an odd integer $k\in\mathbb{Z}$ such that $\R(z) \sim Cz^{k}$ for $z\rightarrow \infty$, and thus $\R(z) \not\equiv 1$.
\end{proof}
However, as shown below, the above barrier for $L$-stable Runge-Kutta methods can be circumvented by applying an appropriate postprocessor to \eqref{eq:XnAB} of the form
\begin{equation}
\overline X_n = \C(z) X_n + \D(z)\sqrt{h} \sigma \xi_n,\qquad z=-\lambda h.
\end{equation}
We will see that this feature serves as a crucial ingredient in Section \ref{sec:SPDE} to achieve high order in the SPDE case.

\begin{proposition} \label{pro:exactLstable}
The method \eqref{eq:newmeth} applied to the SDE test problem \eqref{eq:OU} with $f_1(x)=-\lambda x$ and $f_2(x)=0$ is such that the scheme $X_n\mapsto X_{n+1}$ is $L$-stable and $\overline X_n$ is exact for sampling the invariant measure, {\it i.e.} for all test functions $\phi$ and all timesteps~$h$,
$
\limas_{M \rightarrow \infty}
\frac{1}{M+1}\sum_{n=0}^{M}\phi(\overline{X}_n)
=\lim_{n\rightarrow +\infty} \IE(\phi(\overline{X}_n)) = \int_{\IR}\phi(y)\rho_\infty (y)dy.
$
\end{proposition}
\begin{proof}
The method is $L$-stable because it has the same stability function $\A(z)=\frac{1}{1-z}$ as the linearized implicit Euler method \eqref{eq:v_n}.
A calculation using the above notations yields
\begin{equation} \label{eq:defRbar}
\lim_{n\rightarrow \infty} \mathbb{E}(\overline X_n^2)  = \frac{\sigma^2}{2\lambda} \overline{\R}(z), \quad \mbox{ where } 
\overline{\R}(z)=\C(z)^2\R(z)-2z\D(z)^2,
\end{equation}
with $z=-\lambda h$ and $\R(z)$ given by \eqref{eq:defR}. Noting that
$
\mathcal A(z) = \B(z) = ({1-z})^{-1},\  \C(z) = 1, \  \D(z) = \frac12({{1-z/2}})^{-1/2},
$
yields $\overline{\R}(z)\equiv 1$, which proves that the method is exact.
\end{proof}

%
%
%

%
\begin{remark} \label{rem:noncommute}
Relaxing the assumption that the matrices $A$ and $Q$ commute, the scheme \eqref{eq:newmeth}
can be adapted to remain exact for sampling the invariant measure when applied to the linear problem
$dX=AXdt + \sigma Q^{1/2}dW$ without nonlinearity ($f_1(x)=Ax,f_2(x)=0$).
In fact, the definition of matrix $J_3$ should be modified, and obtained 
solving a system of Lyapunov equations.
\end{remark}

\subsection{Construction of the integrator~\eqref{eq:newmeth}}


\label{sec:nonlineardimfinie}


We explain in this section how we construct method \eqref{eq:newmeth} with high order of accuracy for the invariant measure
by using the idea of postprocessing from \cite{Vil15}. 
Consider a system of SDEs in $\IR^N$ of the form
\begin{equation} \label{eq:e:main_sde}
dX(t) = f_0(X(t)) dt + \sigma dW(t),\qquad X(0)=X_0,
\end{equation}
where $\sigma>0$ and $\bigl(W(t)\bigr)_{t\in\IR^+}$ is a standard $d$-dimensional Wiener process.

\begin{ass} \label{ass:ergodic} We assume the following.
\begin{enumerate}
\item $f_0$ is of class $C^{\infty}$, with bounded derivatives of any order, and
there exists a potential function $V_0:\IR^N\to \IR$ such that $f_0=-\nabla V_0$; 

\item there exist $C,\beta>0$ 
such that for all $x\in \IR^N$,
$x^T f_0(x) \leq -\beta x^Tx+C$.
\end{enumerate}
\end{ass}
Then, system \eqref{eq:e:main_sde} is ergodic with a unique invariant measure $\mu_\infty$ (see e.g. \cite{Has80}) given by
$$\mu_\infty(dy)=\rho(y)dy, \quad \mbox{with } \rho(y)=\frac{1}{Z}\exp(-\frac{2}{\sigma^2}V_0(y)).$$
Moreover the solution $X(t)$ of \eqref{eq:e:main_sde} satisfies an exponential ergodicity property
analogous to \eqref{eq:expergoSPDE}.



The following theorem is the main result in \cite{Vil15} where postprocessed integrators for SDEs are introduced. It permits to improve the accuracy
of a method of weak order $q$ to order $r=q+1$ for the invariant measure. 
Notice that this theorem remains valid for general classes of SDEs with additive or multiplicative noise in multiple dimensions.
The error estimates in \cite{Vil15} rely on classical results from 
Talay and Tubaro \cite{TaT90} and Milstein \cite{Mil86} based on the backward Kolmogorov equation
(see also \cite[Chap.\ts 2.2,\ts 2.3]{MiT04}). 
In this section, we denote by $\CC_P^\infty(\IR^N,\IR)$ the set of $\CC^\infty$ functions whose derivatives
up to any order have a polynomial growth. 
We denote by $\mathcal{L}$ the infinitesimal generator of \eqref{eq:f1f2} where we set $f_0=f_1+f_2$: for any test function $\phi\in\CC_P^\infty(\IR^N,\IR)$,
\begin{equation} \label{eq:generator}
\mathcal{L}\phi= f_0\cdot \nabla\phi+\frac{\sigma^2}{2}\Delta \phi,
\end{equation}
where $\nabla \phi$ denotes the gradient of $\phi$ and $\Delta \phi$ the Laplacian of $\phi$.

\begin{theorem} \cite[Theorem\ts 4.1]{Vil15} \label{thm:main}
 Under Assumption \ref{ass:ergodic},
let $X_{n}$ be an ergodic numerical solution of \eqref{eq:e:main_sde} with bounded moments of any order $M\in \IN^*$, {\it i.e.} 
\begin{equation} \label{eq:boundedmoments}
\IE(\|X_n\|^M) \leq C_M
\end{equation}
for all $n\geq 0$, where $C_M$ is independent of $h,n$.
Assume further that the scheme has local weak order $q\geq 1$ {\it i.e.} it satisfies for all initial condition $X_0=x$ and all
$h$ sufficiently small,
\begin{equation}  
\label{weak:conv_loc}
|\IE(\phi(X_1)) -  \IE(\phi(X(h)))| \leq C(x,\phi) h^{q+1},
\end{equation}
for all $\phi\in\CC_P^\infty(\IR^N,\IR)$, where $x\mapsto C(x,\phi)$ has a polynomial growth with respect to $x$.\looseness-3

Assume also that the numerical solution admits a weak Taylor series
expansion of the form
\begin{equation} \label{eq:numin_taylor_expansion_formal}
\IE(\phi(X_1))=\phi(x)+h\mathcal{A}_0\phi(x)+h^2 \mathcal{A}_1\phi(x)+\ldots,
\end{equation}
for all $\phi \in \CC_P^\infty(\IR^N,\IR)$,
where $\mathcal{A}_i:\CC_P^\infty(\IR^N,\IR) \rightarrow \CC_P^\infty(\IR^N,\IR),~i=0,1,2,\ldots$ are linear differential operators with smooth coefficients. 
Let $G_n$ denote independent and identically distributed random maps in $\IR^N$,  independent of $\{X_{j}\}_{j\leq n}$,
with $\overline X_n = G_n(X_n)$ having bounded moments of any order,
and satisfying a weak Taylor expansion of the form
\begin{equation} \label{eq:phiG}
\IE(\phi(G_n(x))) = \phi(x) + h^q \mathcal{\overline A}_q \phi(x) + \bigo(h^{q+1}),
\end{equation}
for all $\phi \in \CC_P^\infty(\IR^N,\IR)$, where the constant in $\bigo$ has a polynomial growth with respect to $x$.
Assuming further
\begin{equation*} 
(\mathcal{A}_{q} + [\mathcal{L},\mathcal{\overline A}_q])^*\rho = 0,
\end{equation*}
where we use the commutator notation $[A,B]=AB-BA$.
Then the postprocessor $\overline X_n = G_n(X_n)$ yields an approximation of order $r=q+1$ for the invariant measure, 
\begin{eqnarray}
\left|\limas_{M \rightarrow \infty} \frac{1}{M+1}\sum_{n=0}^{M} \phi(\overline{X}_{n})-\int_{\IR^N}\phi(y)d\mu_\infty(y) \right|
&\leq& C(\phi)h^{q+1},\\
\label{eq:numexp}
\left|\IE(\phi(\overline{X}_n)) - \int_{\IR^N} \phi(y) d\mu_\infty(y) \right|
&\leq& K(\phi,x)e^{-\lambda t_n} + C(\phi)h^{q+1}
\end{eqnarray}
for all $\phi\in\CC^\infty_P(\IR^N,\IR)$ with $t_n=nh$, where $\lambda,K(\phi,x),C(\phi)$ are independent of $n$ and $h$ assumed small enough.
\end{theorem}
\begin{remark}
We emphasize that the boundedness of moment condition \eqref{eq:boundedmoments} can be easily proved for Runge-Kutta type methods, such as the proposed method \eqref{eq:newmeth},
following the methodology of \cite{Mil86} (see also \cite[Chap.\ts 2.2]{MiT04}) using the global Lipschitz continuity of the SDE fields. 
We also refer to \cite{Vil15} where this assumption is discussed in the context of postprocessed integrators
for SDEs. 
Notice also that the Lipschitz condition on the SDE fields and the ergodicity
assumption on the numerical method could be relaxed using the concept of rejecting exploding trajectories, as introduced in \cite{MT05}, see also \cite{MT07} in the context of ergodic SDEs.
\end{remark}

We are now in position to state the main result of this section.
We show that the new method \eqref{eq:newmeth} satisfies the assumptions of Theorem \ref{thm:main} with $q=1$, and thus has order $r=2$ of accuracy for the invariant measure of ergodic SDEs.

\begin{theorem} \label{thm:order2}
 Under Assumption \ref{ass:ergodic}, consider the method \eqref{eq:newmeth} with postprocessor $\overline X_n$. Assume that $X_n$ is ergodic when applied to the system \eqref{eq:f1f2}. Then $\overline X_n$ has order two of accuracy for the invariant measure, precisely,
\begin{eqnarray} \label{eq:th1}
\left| \limas_{M \rightarrow \infty} \frac{1}{M+1}\sum_{k=0}^{M} \phi(\overline X_{k})-\int_{\IR^N}\phi(y)d\mu_\infty(y) \right| &\leq& C(\phi) h^{2}
\\ \label{eq:th2}
\left|\IE(\phi(\overline X_n)) - \int_{\IR^N} \phi(y) d\mu_\infty(y) \right| &\leq& K(\phi,x)e^{-\lambda t_n} + C(\phi) h^{2}, 
\end{eqnarray}
for all $\phi\in\CC^\infty_P(\IR^N,\IR)$ with $t_n=nh$, where $\lambda,K(\phi,x),C(\phi)$ are independent of $n$ and $h$ assumed small enough.
\end{theorem}

The proof of Theorem \ref{thm:order2} relies on the following lemma where conditions of order two of accuracy are derived for a perturbation for the linearized Euler method. The proof of this lemma is postponed to the Appendix.
\begin{lemma} \label{lemma:order2}
Consider the following modification of the linearized Euler scheme for \eqref{eq:f1f2},
\begin{eqnarray}
Y_{n+1} &=& Y_n + h f_1\left(Y_{n+1} + a_1 \sigma \sqrt h \xi_n\right) + h f_2(Y_{n} + a_2 \sigma \sqrt h \xi_n)  + (I+a_3hf_1'(Y_n)) \sigma \sqrt h \xi_n \nonumber \\
\overline Y_n &=& Y_n + b_1 h f_1(\overline Y_n) +  b_2 h f_2(Y_{n})  + c \sigma \sqrt h \xi_n.  \label{eq:abc}
\end{eqnarray}
where $a_1,a_2,a_3,b_1,b_2,c$ are fixed real coefficients.
If the following conditions hold,
\begin{align}
&a_1^2+2 a_1 +b_1 -c^2 = 0, \qquad
a_1 +a_3 +\frac14 +b_1 -c^2 = 0, \qquad
a_2^2 +b_2 -c^2 = 0, \nonumber\\
&-\frac14 + a_2 +b_2  -c^2 = 0, \qquad
(b_2-b_1)[f_2,f_1] =0, \label{eq:ordercond}
\end{align}
then, assuming the ergodicity of $Y_n$ in \eqref{eq:abc}, the postprocessed scheme \eqref{eq:abc} satisfies the assumptions of Theorem \ref{thm:main} with $q=1$, and $\overline Y_n$ has order two of accuracy for the invariant measure, {\it i.e.} it satisfies \eqref{eq:th1},\eqref{eq:th2} (with $\overline X_n$ replaced by $\overline Y_n$).
\end{lemma}

Note that the Lie bracket $[f_2,f_1]=f_2'f_1-f_1'f_2$ involved in the second order conditions \eqref{eq:ordercond}
vanishes only when the flows associated to the fields $f_1,f_2$ commute, which is not true in general. We thus impose $b_1=b_2$. Still, the system \eqref{eq:ordercond} has infinitely many solutions.
Setting $b_1=b_2=0$ for simplicity of the postprocessor, two solutions remain.
Choosing the solution which minimizes the absolute value of $a_1$ and $a_3$, we obtain the following choice of coefficients for the order two scheme \eqref{eq:abc},
\begin{equation} \label{eq:defabc}
a_1 = -a_3 = \frac{-2+\sqrt5}2, \qquad a_2 = c = \frac 12,\qquad b_1=b_2=0.
\end{equation}

\begin{proof}[Proof of Theorem \ref{thm:order2}]
It is sufficient to prove that
the method \eqref{eq:newmeth} satisfies the same expansions \eqref{eq:numin_taylor_expansion_formal} and 
\eqref{eq:phiG} with $q=1$ as method \eqref{eq:abc},\eqref{eq:defabc}  (with the same differential operators $\mathcal{A}_0=\mathcal{L}$, $\mathcal{A}_1$, and $\mathcal{\overline A}_1$). 
Indeed,  applying Theorem \ref{thm:main}  with $q=1$  to the scheme \eqref{eq:newmeth}, we deduce that method \eqref{eq:newmeth} also has second order of accuracy for the invariant measure, which concludes the proof of Theorem \ref{thm:order2}. 

To recover the scheme \eqref{eq:newmeth} from \eqref{eq:abc},\eqref{eq:defabc}, in the first line of \eqref{eq:abc} one has to
  replace $\xi_n$ with $J_{n,2}\xi_n=(I+\bigo(h))\xi_n$ in the arguments of $f_1,f_2$ and also to substitute $I+a_3 h  f_1'(x)$ with
$$
\big(\frac{1-\sqrt2+\sqrt{5}}2 J_{n,1}^{-1} + \frac{1+\sqrt2-\sqrt{5}}2 \big)J_{n,2} = I+a_3 h f_1'(x) + \bigo(h^2);
$$
in the second line of \eqref{eq:abc}, one has to replace $\xi_n$ with $J_{n,3}\xi_n=(I+\bigo(h))\xi_n$. We obtain that the difference between one step of \eqref{eq:newmeth} and one step of \eqref{eq:abc},\eqref{eq:defabc} with initial condition $X_0=Y_0=x$ has the form $X_1-Y_1 = R(x)\xi h^{5/2} + \bigo(h^3)$. Using $\IE(\xi)=0$, we deduce $\IE(\phi(X_1))-\IE(\phi(Y_1))=\bigo(h^3)$, while
$\IE(\phi(\overline X_0))-\IE(\phi(\overline Y_0))=\bigo(h^2)$.
\end{proof}


It can be seen from the proof of Theorem \ref{thm:order2} that the operator $J_{n,2}$ in front of $\xi_n$ and the operator $J_{n,3}$ in the definition of the method  \eqref{eq:newmeth} have no influence on its order two of accuracy for the invariant measure
in finite dimension. In infinite dimension, however, these operators play an important role for the well-posedness, the stability and the accuracy of the scheme in the SPDE case presented in Section~\ref{sec:SPDE}.


\section{Analysis in the SPDE case}
\label{sec:SPDE}

\subsection{Abstract setting and assumptions}
\label{sec:hyp}

The state space in the SPDE case is an infinite dimensional separable Hilbert space  $\H$, for which we denote by $\langle \cdot, \cdot \rangle$ the scalar product, and by $|\cdot|$ the associated norm.
%
%
%
%
%
%
%
Consider the linear operator $A$ involved in the parabolic SPDE \eqref{eq:SPDE}.
Recall that we assume that $-A$ is an unbounded self-adjoint linear operator with eigenvalues
$0<\lambda_1\leq \ldots\leq \lambda_{p}\leq  \lambda_{p+1} \leq \ldots$, such that $\lambda_p\rightarrow +\infty,$ when $p\rightarrow +\infty$,
and associated normalized eigenvectors $e_p$ (such that $Ae_p=-\lambda_p e_p$), which form a complete orthonormal system in $\H$.

For any ${\s}\in\IR^+$, we classically define the unbounded linear operator $(-A)^{\s/2}$ from $\H$ to $\H$ and its domain $\H^{\s}\subset \H$ as follows:
$$
(-A)^{\s/2}u=\sum_{p=1}^{+\infty}\langle u,e_p\rangle \lambda_{p}^{\s/2}e_p \quad \mbox{for all } u\in \H^{\s}=\{ u\in \H ~:~ |u|_{\s}^{2}=\sum_{p=1}^{+\infty}|\langle u,e_p\rangle|^2\lambda_{p}^{\s}<+\infty\}.
$$
We also define the bounded linear operator $(-A)^{-\s/2}$ 
and the semi-group $\bigl(e^{tA}\bigr)_{t\in\IR^+}$, both as linear operators from $\H$ to $\H$ by 
$$
(-A)^{-\s/2}u=\sum_{p=1}^{+\infty}\langle u,e_p\rangle \lambda_{p}^{-\s/2}e_p, \quad \quad
e^{tA}u=\sum_{p=1}^{+\infty}\exp(-t\lambda_p)\langle u,e_p\rangle e_p.
$$
It is straightforward that for any $t\in(0,+\infty)$ we have $e^{tA}\in \mathcal{L}(\H,\H)$ -- the space of bounded linear operators from $\H$ to $\H$, endowed with the norm denoted by $\|\cdot\|$ -- with $\|e^{tA}\|\leq \exp(-\lambda_1 t)$. Moreover, the following regularization property holds true:
$\| (-A)^{\s/2}e^{tA}\|\leq \frac{C_{\s}}{t^{\s/2}}$
where $C_{\s}^{2}=\sup_{r\in \IR^+} \exp(-2r)r^{\s}\in(0,+\infty)$.

\paragraph{Covariance operator}
We assume that $Q$ is a bounded, non-negative self-adjoint linear operator from $\H$ to $\H$, which satisfies
$Qe_p=q_pe_p$
for any $p\in\IN^*$, where the eigenvalues $\bigl(q_p\bigr)_{p\in\IN^*}$ form a bounded sequence of non-negative real numbers. We assume that condition \eqref{eq:conditionTrace} is satisfied. The $Q$-Wiener process in \eqref{eq:SPDE} is then defined as follows: for any $t\geq0$,
\begin{equation}\label{eq:Q-Wiener}
W^{Q}(t)=\sum_{p=1}^{+\infty}\sqrt{q_p}\beta_p(t)e_p,
\end{equation}
where $\bigl(\beta_p\bigr)_{p\in\IN^*}$ is a sequence of independent standard scalar Wiener processes on an underlying probability space $(\Omega,\mathcal{F},\mathbb{P})$.


Notice that the operators $A$ and $Q$ commute: $AQu=QAu$ for any $u\in \H^2$. This property is often assumed in the literature, and simplifies the analysis of the order of convergence made in Section~\ref{sec:analysis_order}. Nevertheless several arguments (especially Proposition~\ref{pro:ergo_cont} and the results of Section~\ref{sec:timemeth}) do not require this property; in particular, the scheme remains well-defined in the non-commuting case.

%
%

\paragraph{Nonlinearity}
The nonlinear coefficient $F$ is assumed to be a Lipschitz continuous function from $\H$ to $\H$, with Lipschitz constant $L$ satisfying the dissipation condition
\begin{equation}\label{ass:L_mu}
{L}=\sup_{u^1\neq u^2 \in \H}\frac{|F(u^1)-F(u^2)|}{|u^1-u^2|}  < \min_{p\in\IN^*}\lambda_p = \lambda_1.
\end{equation}
This condition ensures ergodicity of the continuous-time process (Proposition~\ref{pro:ergo_cont}) and of its time-discretized approximations (Proposition~\ref{pro:invar}). A typical example of such a Lipschitz function on $\H=L^2(\mathcal{D})$ -- where $\mathcal{D}$ is an open smooth bounded domain in $\IR^d$ -- is the Nemytskii operator
$F: u\mapsto f \circ u$, where $f:\IR\to\IR$ is a globally Lipschitz function. 
Note that $F=-DV$ is the derivative of the potential function $V:L^2(\mathcal{D})\rightarrow \IR$ 
where $V(u) =- \int_0^1\int_{\mathcal{D}}    u(x) f(\theta u(x)) dx d\theta$.

%
%


Under the above hypotheses, and assuming the trace condition \eqref{eq:conditionTrace}, 
the process $\bigl(u(t)\bigr)_{t\in\IR^+}$ takes values in $\H^{\s}$ for any $\s<\overline{\s}$, and we recall without proof the following result of exponential convergence to a unique invariant distribution, see e.g. \cite{DPZ96} for general results, and \cite[Section 3.1.1]{D13}.
\begin{proposition}\label{pro:ergo_cont}
Assume \eqref{eq:conditionTrace} and the above hypotheses. 
Then, the process $\bigl(u(t)\bigr)_{t\in\IR^+}$ solution of \eqref{eq:SPDE} admits a unique invariant probability distribution $\mu_{\infty}$ on $\H$.
Moreover for all $\s<\overline{\s}$, 
\begin{equation*}
\int_{\H}|u|_{\s}^{2}\mu_{\infty}(du)<+\infty,
\end{equation*}
and for all $\phi:\H\rightarrow \IR$ Lipschitz continuous, and all $t>0$,
$$\Big| \mathbb{E}\bigl[ \phi(u(t)) \bigr]-\int_{\H}\phi(v)\mu_{\infty}(dv) \Big| \leq 
C(\phi,u_0)e^{-(\lambda_1-{L})t},$$
where $C(\phi,u_0)$ is independent of $t$.
\end{proposition}

Condition \eqref{ass:L_mu} is crucial for the proof of the uniqueness of the numerical invariant distributions established in the next section: 
we compare the solutions starting from different initial conditions and driven with the same noise process, and show an exponential contraction similar to the result 
of Proposition~\ref{pro:ergo_cont}. 
Notice that weaker conditions than \eqref{ass:L_mu} are known in the literature (see e.g. \cite{D13} and references therein) 
to ensure the ergodicity and the exponential convergence of \eqref{eq:SPDE} -- for instance when $F$ is bounded and Lipschitz continuous with no size restriction on ${L}$.



\subsection{Stability and ergodicity of the integrator for SPDEs}
\label{sec:timemeth}

In this section, we prove the existence and uniqueness of invariant distributions for the time-discretized  processes defined by the numerical method, for any time-step size $h>0$, in a general setting. 
Notice that the results of this section do not require the gradient assumption $F=-DV$.
The results are analogous to classical results for the $\theta$-method  in the context of stiff SDEs \cite[Theorem 3.1]{BRK11} and
for the linearized implicit Euler method \eqref{eq:u_n} in the context of SPDEs as studied e.g. in \cite[Remark\ts4.8]{B14}.
It is a key observation here to exploit that
the sequence $(u_n,\overline{u}_{n-1})_{n\in\IN}$ defining the new scheme \eqref{eq:u_n},\eqref{eq:u_n_bar} is a Markov chain on the product space $\H\times\H$. The initial condition $(u_0,\overline{u}_{-1})$ is given by $u_0=u(0)=u_0$ and an arbitrary $\overline{u}_{-1}\in \H$ which plays no role in the dynamics, since $(u_{n+1},\overline{u}_n)$ depends only on $u_n$ and $\xi_{n}^{Q}$, not on $\overline{u}_{n-1}$.

In the following proposition,
we state uniform bounds -- with respect to $n\in \IN$ and $h\in(0,1)$ -- on first-order moments for the norm $|\cdot|_{\s}$  for $\s<\overline{\s}$.

\begin{proposition}\label{pro:moments}
Assume the hypotheses of Section \ref{sec:hyp} and consider the scheme \eqref{eq:u_n},\eqref{eq:u_n_bar}. 
For all $\s\in[0,\overline{\s})$, assuming $u_0\in \H^{\s}$, there exists a constant $C_{\s}\in(0,+\infty)$ such that for all $h\in(0,1)$,
$$
\sup_{n\in\IN}\IE\big|u_n\big|_{\s}\leq C_{\s}(1+|u_0|_{\s}), \qquad \sup_{n\in\IN}\IE\big|\overline u_n\big|_{\s}\leq C_{\s}(1+|u_0|_{\s}).$$
\end{proposition}

\begin{proof}
Thanks to \eqref{eq:conditionTrace}, $\text{Trace}\bigl(J_{i}QJ_{i}\bigr)<+\infty$ for $i\in\left\{1,2,3\right\}$, and thus $u_n$ and $\overline{u}_n$ are well-defined in $\H$ for all $n\in\IN$. 
The contributions of the drift part and of the stochastic perturbation are treated separately: we introduce the auxiliary process $\bigl(\ell_n\bigr)_{n\in\IN}$, as the solution of the following equation
\begin{equation*}
\ell_{n+1}=J_1\ell_{n}+\sigma\sqrt{h}\bigl(\frac{\sqrt{2}-1}{2}J_1+\frac{3-\sqrt{2}}{2}I\bigr)J_2\xi_{n}^{Q},
\end{equation*}
with $\ell_0=0$. Set $d_n=u_n-\ell_n$ for all $n\in\IN$; then $d_0=u_0$ and
$$d_{n+1}=J_1d_n+hJ_1F(d_n+\ell_n+\frac{1}{2}J_2\sigma\sqrt{h}\xi_{n}^{Q}).$$
The quantity $\ell_n$ satisfies for all $n\in\IN^*$ the identity
$$\ell_{n}=\sigma\sqrt{h}\bigl(\frac{\sqrt{2}-1}{2}J_1+\frac{3-\sqrt{2}}{2}I\bigr)J_2J_{1}^{-1}\sum_{k=0}^{n-1}\bigl(J_1\bigr)^{n-k}\xi_{k}^{Q}.$$
Observe that $\bigl(\frac{\sqrt{2}-1}{2}J_1+\frac{3-\sqrt{2}}{2}I\bigr)J_2J_{1}^{-1}$ is a bounded linear operator from $\H^{\s}$ to $\H^{\s}$ with norm less than $\frac{2}{3-\sqrt{2}}$, for $\s\in [0,\overline{s})$. Since the Gaussian random variables $\bigl(\xi_{k}^{Q}\bigr)_{k\in\IN}$ are independent, for all $0\leq \s\leq \overline{\s}$,
\begin{align*}
\IE\big|\ell_n\big|_{\s}^{2}&\leq \frac{2\sigma^2 h}{3-\sqrt{2}} \sum_{k=0}^{n-1}\IE\big|\bigl(J_1\bigr)^{n-k}\xi_{k}^{Q}\big|_{\s}^{2}\leq \frac{2\sigma^2 h}{3-\sqrt{2}} \sum_{k=1}^{+\infty}\sum_{p=1}^{+\infty}\frac{q_p\lambda_{p}^{\s}}{(1+\lambda_p h)^{2k}}\\
&\leq \frac{2\sigma^2}{3-\sqrt{2}} \sum_{p=1}^{+\infty}q_p\lambda_{p}^{\s-1}\frac{\lambda_p h}{(1+\lambda_p h)^2-1}\leq \frac{\sigma^2}{3-\sqrt{2}}{\rm Trace}\bigl((-A)^{-1+\s}Q\bigr).
\end{align*}
Now thanks to \eqref{ass:L_mu}, straightforward computations show that 
\begin{align*}
\IE|d_{n}|&\leq \frac{1+hL}{1+\lambda_1 h}\IE|d_{n-1}|+\frac{h}{1+\lambda_1 h}|F(0)|+\frac{Lh}{1+\lambda_1 h}\bigl(\IE|\ell_{n-1}|+\IE\big|\frac{1}{2}J_2\sigma\sqrt{h}\xi_{n-1}^{Q}\big|\bigr)\\
&\leq \frac{(1+hL)^{n}}{(1+\lambda_1 h)^{n}}|u_0|+\frac{|F(0)|}{\lambda_1-L}+\frac{L}{\lambda_1-L}\bigl(\sup_{k\in\IN}\IE|\ell_k|+\frac{1}{3-\sqrt{2}}\sigma\bigl({\rm Trace}\bigl((-A)^{-1}Q\bigr)\bigr)^{1/2}\bigr),
\end{align*}
with $ \frac{(1+hL)}{(1+\lambda_1 h)}\leq 1- \frac{\lambda_1-L}{1+\lambda_1 h}h \leq \exp\left(-\frac{(\lambda_1-L)}{1+\lambda_1 h} h\right)<1$.


As a consequence the claim follows for $(u_n)_{n\in\IN}$ in the case $\s=0$. In particular, for some constant $C\in(0,+\infty)$ it comes that $\sup_{n\in\IN}\IE\big|F(u_n+\frac{1}{2}J_2\sigma\sqrt{h}\xi_{n}^{Q})\big|<C(1+|u_0|)$.

The case $\s\in(0,\overline{\s})$ is treated using the estimates of Lemma $3.2$ in \cite{B14}: for $n\in\IN$
\begin{align*}
\IE\big|d_n\big|_{\s}&=\IE\big|(-A)^{\s/2}(J_1)^n u_0+h\sum_{k=0}^{n-1}(-A)^{\s/2}(J_1)^{n-k}F(u_k+\frac{1}{2}J_2\sigma\sqrt{h}\xi_{k}^{Q})\big|\\
&\leq \frac{1}{(1+\lambda_1 h)^n}|u_0|_{\s}+hC(1+u_0|)\sum_{k=1}^{n}\bigl(\frac{\mathds{1}_{kh\leq 1}}{(kh)^{\s/2}}+\mathds{1}_{kh>1}\frac{C_{\s}}{(1+\lambda_1 h)^{k-\lfloor 1/h \rfloor}}\bigr)\\
&\leq C_{\s}(1+|u_0|_{\s}).
\end{align*}

Finally, for $n\in\IN$, 
$$\IE\big|\frac{1}{2}J_3\sigma \sqrt{h}\xi_{n}^{Q}\big|_{\s}^{2}=\frac{\sigma^2 h}{4}{\rm Trace}\bigl((I-\frac{h}{2}A)^{-1}(-A)^{\s}Q\bigr)\leq \frac{\sigma^2}{2}{\rm Trace}\bigl((-A)^{-1+\s}Q\bigr).$$
This concludes the proof, since $\overline{u}_n=u_n+\frac{1}{2}J_3\sigma \sqrt{h}\xi_{n}^{Q}$.
\end{proof}

We now state the following existence and uniqueness result for the invariant distribution of the Markov chain $\bigl(u_n,\overline{u}_{n-1}\bigr)_{n\in\IN}$.
We prove that the convergence is exponentially fast, in contrast to the trapezoidal method \eqref{eq:CN} which is not $L$-stable and for which such exponential estimate does not hold in general for stiff problems (even in finite dimension). 
\begin{proposition}\label{pro:invar}
Assume the hypotheses of Section~\ref{sec:hyp}. For any $h\in(0,1)$, the $\H\times \H$-valued Markov chain $\bigl(u_n,\overline{u}_{n-1}\bigr)_{n\in\IN}$ admits a unique invariant distribution ${\rm M}_{\infty}^{h}$ 
in $\H\times \H$, with marginals in $\H$ denoted by $\mu_{\infty}^{h}$ and $\overline{\mu}_{\infty}^{h}$, respectively.

Moreover, the convergence of the distributions to equilibrium is exponentially fast: for all Lipschitz test function $\varphi:\H\rightarrow \mathbb{R}$, and all $t_n=nh$,
\begin{gather*}
\Big|\IE\varphi(u_n)-\int_{\H}\varphi d\mu_{\infty}^{h}\Big| +
\Big|\IE\varphi(\overline{u}_n)-\int_{\H}\varphi d\overline{\mu}_{\infty}^{h}\Big|\leq C(\varphi,|u_0|)
\exp\left(-\frac{(\lambda_1-L)}{1+\lambda_1 h} t_n\right).
\end{gather*}
\end{proposition}

\begin{proof}
\textit{Existence.}
The semi-group $\bigl(P^n\bigr)_{n\in\IN}$ on $\H\times\H$ generated by the Markov chain $\bigl(u_n,\overline{u}_{n-1}\bigr)_{n\in\IN}$ satisfies the Feller property: for any $n\in\IN$, for any bounded continuous test function $\phi:\H\times\H\rightarrow \mathbb{R}$, the map
$(u_0,\overline{u}_{-1})\mapsto P^n\phi(u_0,\overline{u}_{-1})=\IE\phi(u_n,\overline{u}_{n-1})$
is continuous.
The claim then follows from the standard Krylov-Bogoliubov criterion (see Section~3.1 in \cite{DPZ96}): given an arbitrary initial condition $(u_0,\overline{u}_{-1})\in\H\times\H$, if ${\rm M}_n$ denotes the law of $(u_n,\overline{u}_{n-1})$, then
\begin{itemize}
\item $\bigl(\frac{1}{n+1}\sum_{k=0}^{n}{\rm M}_k\bigr)_{n\in\IN}$ is a {\em tight} sequence of probability distributions on $\H\times\H$ -- as a consequence of Proposition~\ref{pro:moments} combined with the Markov inequality, and of the observation that for any $\s\in(0,\overline{\s})$ and any $R>0$, the set $\left\{|u|_{\s}\leq R , |\overline{u}|_{\s}\leq R\right\}$ is a compact subset of $\H\times\H$.
\item every subsequence limit point ${\rm M}$ is an invariant distribution for the semi-group.
\end{itemize}

\textit{Uniqueness.}
Consider two initial conditions $u_0^{1},u_0^{2}\in \H$, as well as $\overline{u}_{-1}^{1},\overline{u}_{-1}^{2}\in\H$ and the associated processes $\bigl(u_n^{i}\bigr)_{n\in\IN}$ and $\bigl(\overline{u}_n^{i}\bigr)_{n\in\IN}$, for $i=1,2$, defined by \eqref{eq:u_n}, \eqref{eq:u_n_bar}, and driven by a unique noise process $\bigl(\xi_n^Q\bigr)_{n\in\IN}$. 

Then by Lipschitz continuity of $F$, and using the cancellations of several noise terms,  computations similar to those of the proof of Proposition~\ref{pro:moments} yield for any $n\in\IN$ the almost sure contraction property
\begin{equation*}
\big|\overline{u}_{n}^{1}-\overline{u}_{n}^{2}\big| = \big|u_{n}^{1}-u_{n}^2\big|\leq \frac{1+Lh}{1+\lambda_1 h}\big|u_{n-1}^1-u_{n-1}^2\big|
\leq \exp\left(-\frac{(\lambda_1-L)}{1+\lambda_1 h} t_{n}\right) \big|u_0^{1}-u_0^{2}\big|. 
\end{equation*}
Finally, taking $(u_0^2,\overline{u}_{-1}^2)$ random, independent of the noise process $\bigl(\xi_n^Q\bigr)_{n\in\IN}$ and distributed according to an ergodic invariant distribution ${\rm M}_{\infty}^{h}$ gives the exponential convergence and the uniqueness properties.
\end{proof}

\subsection{Analysis of the order of convergence: a simplified linear case}\label{sec:analysis_order}

It is shown in \cite{B14} (for $d=1$ and $\overline{\s}=1/2$, associated with space-time white noise $Q=I$),
that the standard linearized implicit Euler scheme \eqref{eq:u_n} has order $r=1/2-\eps$
for all $\eps\in(0,1/2)$ for the approximation of the invariant distribution $\mu_{\infty}$ of \eqref{eq:SPDE}.
In this section, we show that the postprocessed scheme has the improved order of convergence $\overline s+1-\eps$.
Since the techniques from Section~\ref{sec:finitedim} do not extend straightforwardly to the SPDE case, we only focus on a simplified case, where the nonlinear coefficient $F$ is replaced with a bounded linear operator. 
Numerical experiments of Section~\ref{sec:num} show that the higher order is preserved for various examples of nonlinearities $F$.
%
%
%
%
%
%

In addition to the hypotheses of Section \ref{sec:hyp}, assume that the coefficient $F$ is given by a linear mapping: for any $u\in\H$,  $F(u)=Bu$ where $B\in\mathcal{L}(\H)$ satisfies $Be_p=-b_pe_p$ for all $p\in\IN^*$, with real eigenvalues $b_p\in(-\lambda_1,\lambda_1)$ (due to condition \eqref{ass:L_mu}):
\begin{equation}\label{eq:SPDE_B}
du(t)=Au(t)dt+Bu(t)dt+\sigma dW^Q(t) \quad , \quad u(0)=u_0.
\end{equation}
In this situation, the components $\langle u(t),e_p\rangle$ in the basis $\{e_p\}_{p\in \mathbb{N}^*}$ of the solution $\bigl(u(t)\bigr)_{t\in\IR^+}$ of the SPDE \eqref{eq:SPDE} are independent Ornstein-Uhlenbeck processes. Similarly, the components in the basis $\{e_p\}_{p\in \mathbb{N}^*}$ of the discrete-time processes $\bigl(v_n\bigr)_{n\in\IN}$ (resp. $\bigl(u_n\bigr)_{n\in\IN}$, resp. $\bigl(\overline{u}_n\bigr)_{n\in\IN}$) are also independent processes.

As a consequence, explicit expressions for the invariant distributions $\mu_{\infty}$, $\nu_{\infty}^{h}$, $\mu_{\infty}^{h}$ and $\overline{\mu}_{\infty}^{h}$ are available: they are centered Gaussian probability measures on $\H$,
\begin{eqnarray*}
\mu_{\infty}&=&\mathcal{N}\bigl(0,Q_{\infty}\bigr) \quad , \quad Q_{\infty}=\frac{\sigma^2}{2}Q(-A-B)^{-1}, \\
\nu_{\infty}^{h}&=&\mathcal{N}\bigl(0,Q_{\infty,\nu}^{h}\bigr) \quad , \quad  Q_{\infty,\nu}^{h}=\frac{\sigma^2}{2}Q(-A-B)^{-1}\bigl(I-h\frac{A-B}{2}\bigr)^{-1},\\  
\mu_{\infty}^{h}&=&\mathcal{N}\bigl(0,Q_{\infty}^{h}\bigr) \quad , \quad \overline{\mu}_{\infty}^{h}=\mathcal{N}\bigl(0,\overline{Q}_{\infty}^{h}\bigr).
\end{eqnarray*}
The expressions for $Q_{\infty}^{h}$ and $\overline{Q}_{\infty}^{h}$ being more complicated are displayed in \eqref{eq:express_Q_h} below.

The following lemma is a key elementary tool in order to exhibit the order of convergence as $h\rightarrow 0$ of the approximating measures towards $\mu_{\infty}$.

\begin{lemma}\label{lem:ordre_Trace}
Let $\pi_j=\mathcal{N}(0,Q_j)$, $j\in\left\{1,2\right\}$ be two Gaussian probability distributions on the Hilbert space $\H$. 
Assume that for all $p\in\IN^*$, $j\in\{1,2\}$ $Q_je_p=q_{j,p}e_p$. 
Let $\varphi\in\mathcal{C}^{2}(\H,\IR)$ satisfy that $\sup_{v\in \H}\|D^2\varphi(v)\|<+\infty$. Then
\begin{equation}
\Big| \int_{\H}\varphi d\pi_2-\int_{\H}\varphi d\pi_1 \Big| \leq \frac{\sup_{v\in \H}\|D^2\varphi(v)\|}{2}\sum_{p=1}^{+\infty}\big|q_{2,p}-q_{1,p}\big|.
\end{equation}
\end{lemma}

\begin{remark}\label{rem:ordre_Trace_optim}
Assume that $q_{2,p}\geq q_{1,p}$. Then the above Lemma~\ref{lem:ordre_Trace} yields optimal orders of convergence. Indeed choosing the test function
$\varphi_{\rm opt}(u)=\exp(-|u|^2)$, 
Lemma $9.5$ in \cite{JK15} yields that
$$\int_{\H}\varphi_{\rm opt}d\pi_1-\int_{\H}\varphi_{\rm opt}d\pi_2\geq \frac{{\rm Trace}\bigl(Q_2-Q_1\bigr)}{\exp\bigl(6{\rm Trace}(Q_2)\bigr)}.$$
This means that the quantity ${\rm Trace}\bigl(Q_2-Q_1\bigr)$ also provides a lower bound for the error between the invariant distributions $\pi_1$ and $\pi_2$.
\end{remark}

\begin{proof}[Proof of Lemma \ref{lem:ordre_Trace}]
Let $\bigl(\gamma_p\bigr)_{p\in\IN^*}$ and
$\bigl(\delta_p\bigr)_{p\in\IN^*}$ be two independent sequences of {\it
i.i.d.} standard real valued Gaussian random variables, centered and
with variance $1$.
Set $X_j=\sum_{p\in\IN^*}\sqrt{q_{j,p}}\gamma_pe_p$, and
$R_{j}=\sum_{p\in\IN^*}\sqrt{\max\bigl((-1)^{j}(q_{1,p}-q_{2,p}),0\bigr)}\delta_pe_p$,
for $j\in\{1,2\}$. Observe that $X_j\sim \pi_j$, and that $X_1+R_{1}$
and $X_2+R_{2}$ have the same Gaussian distribution. This yields
\begin{align*}
&\Big| \int_{\H}\varphi d\pi_2-\int_{\H}\varphi
d\pi_1\Big|=\Big|\IE\bigl[\varphi(X_2)\bigr]-\IE\bigl[\varphi(X_1)\bigr]\Big|\\
&=\Big|\IE\bigl[\varphi(X_2)\bigr]-\IE\bigl[\varphi(X_2+R_2)\bigr]
+\IE\bigl[\varphi(X_1+R_1)\bigr]-\IE\bigl[\varphi(X_1)\bigr]\Big|\\
&\leq \Big|\IE\bigl[\varphi(X_2+R_2)\bigr]-\IE\bigl[\varphi(X_2)\bigr]
-\IE\bigl[D\varphi(X_2).R_2\bigr]\Big|\\
&\phantom{=}+\IE\Big|\bigl[\varphi(X_1+R_1)\bigr]-\IE\bigl[\varphi(X_1)\bigr]
-\IE\bigl[D\varphi(X_1).R_1\bigr]\Big|.
\end{align*}
Indeed, $\IE\bigl[D\varphi(X_j).R_j\bigr]=0$, because $X_j$ and $R_j$ are
independent and $\IE\bigl[R_j\bigr]=0$.
Using a second-order Taylor expansion, we deduce
\begin{equation*}
\Big| \int_{\H}\varphi d\pi_2-\int_{\H}\varphi d\pi_1\Big|\leq
\frac{\sup_{v\in \H}\|D^2\varphi(v)\|}{2}\bigl(\IE
\big|R_{1}\big|^2+\IE\big|R_2\big|^2\bigr),
\end{equation*}
where we note
$\IE\big|R_{1}\big|^2+\IE\big|R_{2}\big|^2=\sum_{p=1}^{+\infty}\bigl|q_{2,p}-q_{1,p}\bigr|$.
\end{proof}

We now explain how Lemma~\ref{lem:ordre_Trace} permits to find the order of convergence of $\overline{\mu}_{\infty}^{h}$ to $\mu_{\infty}$.
We define, for $p\in \IN^*$, the component processes $u_{\cdot}(p)$ and $\overline{u}_{\cdot}(p)$ by projecting on the eigenvector $e_p$: for any $n\in\IN$,
$u_n(p)=\langle u_n,e_p\rangle, \ \overline{u}_{n-1}(p)=\langle \overline{u}_{n-1},e_p\rangle.$
Then \eqref{eq:u_n},\eqref{eq:u_n_bar} applied to \eqref{eq:SPDE_B} rewrites as a system of independent equations, decoupled with respect to $p\in\IN^*$,
\begin{equation}
\begin{gathered}
u_{n+1}(p)=\A(-\lambda_ph,-b_ph)u_n(p)+\sigma \sqrt{h}\sqrt{q_p}\B(-\lambda_ph,-b_p h)\xi_{n,p}\\
\overline{u}_{n}(p)=\C(-\lambda_ph)u_n(p)+\sigma \sqrt{h}\sqrt{q_p}\D(-\lambda_ph)\xi_{n,p},
\end{gathered}
\end{equation}
where $\sqrt{q_p}\xi_{n,p}=\langle\xi_{n}^{Q},e_p\rangle$: thus $\bigl(\xi_{n,p}\bigr)_{n\in\IN,p\in\IN^*}$ are independent standard Gaussian random variables, and the rational functions $\A,\B,\C,\D$ satisfy for any $z\in(-\infty,0)$ and $\beta\in(-1,\min(1,|z|))$,
\begin{equation}
\A(z,\beta)=\frac{1+\beta}{1-z} \quad \B(z,\beta)=\frac{1+\frac{\beta}{2}-\frac{3-\sqrt{2}}{2}z}{(1-z)(1-\frac{3-\sqrt{2}}{2}z)}, \quad
\C(z)=1, \quad \D(z)=\frac{1}{2(1-z/2)^{1/2}}.
\end{equation}
Since for $\beta\in(-1,\min(1,|z|))$ the stability condition $|\A(z,\beta)|<1$ is satisfied, straightforward computations yield
\begin{equation}\label{eq:express_Q_h}
\begin{gathered}
Q_{\infty}^{h}e_p=\lim_{n\rightarrow +\infty}\IE|u_n(p)|^2e_p=\frac{\sigma^2 q_p}{2(\lambda_p+b_p)}\R(-\lambda_p h,-b_p h)e_p\\
\overline{Q}_{\infty}^{h}e_p=\lim_{n\rightarrow +\infty}\IE|\overline{u}_n(p)|^2e_p=\frac{\sigma^2 q_p}{2(\lambda_p+b_p)}\overline{\R}(-\lambda_p h,-b_p h)e_p
\end{gathered}
\end{equation}
where $\R(z,\beta)=\frac{-2(z+\beta)\B(z,\beta)^2}{1-\A(z,\beta)^2}$ and
$$\overline{\R}(z,\beta)=\C(z)^2\R(z,\beta)-2(z+\beta)\D(z)^2=1+\beta z\frac{ P_1(z) \beta + P_2(z)}{(2+\beta-z)P_3(z)}
$$
with polynomial functions $P_1(z) = 10-4\sqrt 2 -(11-6\sqrt 2)z$, $P_2(z) = 20-8\sqrt2 -(44-24\sqrt2)z+(11-6\sqrt2)z^2$ and $P_3(z)=(2-z)(2-(3-\sqrt 2) z)^2$.

The following estimate $\overline{\R}(z,\beta)=1+\bigo(z\beta)$ as $z,\beta\rightarrow 0$ is crucial to obtain an improved order for the convergence of $\overline{\mu}_{\infty}^{h}$ to $\mu_{\infty}$ when $h\rightarrow 0$. 
It is not surprising because the scheme samples exactly the invariant measure of \eqref{eq:SPDE_B} in both cases $A=0$ or $B=0$
(as already shown in Proposition \ref{pro:exactLstable}), equivalently $\overline{\R}(z,0)=\overline{\R}(0,\beta)=1$ for all $z,\beta$. 
\begin{lemma}\label{lem:aux_poly}
For all $z\leq 0$ and all $\beta\in(-1,\min(1,|z|))$, we have
$$
\big|1-\overline{\R}(z,\beta)\big|\leq |z\beta|\frac{(15-6\sqrt{2})}{4(1-z)^2}.$$
\end{lemma}

\begin{proof}
Observe that for all $z\leq 0$, we have $P_2(z)\geq P_1(z) \geq 0$ and $(2+\beta-z)>0$, $P_3(z)>0$. 
Since $({2+\beta-z})^{-1} \leq ({1-z})^{-1}$, we obtain
$$\big|1-\overline{\R}(z,\beta)\big|\leq |z\beta| \frac{P_1(z)+P_2(z)}{(1-z)P_3(z)}.$$
The estimate then follows by observing that $(1-z)\frac{P_1(z)+P_2(z)}{P_3(z)}$ is an increasing function of $z\leq 0$, with maximum
at $z=0$, given by $\frac{P_1(0)+P_2(0)}{P_3(0)} = \frac{(15-6\sqrt{2})}{4}$.
\end{proof}

We are in position to state our main convergence result, which yields the order of convergence $r=\s+1$ for any $\s<\overline{\s}$ for the invariant distribution with postprocessing~$\overline{\mu}_{\infty}^{h}$.
\begin{theorem}\label{pro:ordre_post}
Consider the method \eqref{eq:u_n}-\eqref{eq:u_n_bar} applied to \eqref{eq:SPDE_B}. 
Let $\varphi\in\mathcal{C}^{2}(\H,\IR)$, such that $\sup_{v\in \H}\|D^2\varphi(v)\|<+\infty$. For any $\s<\overline{\s}$ there exists $C_{\s}\in(0,+\infty)$ such that for any $h\in(0,1/L)$ we have
\begin{equation}\label{eq:convhspost}
\Big| \int_{\H}\varphi d\mu_{\infty}-\int_{\H}\varphi d\overline{\mu}_{\infty}^{h} \Big| \leq C_{\s}\sup_{v\in \H}\|D^2\varphi(v)\| \sigma^2 h^{1+\s}.
\end{equation}
Moreover, if $B=0$ ({\it i.e. } $b_p=0$ for all $p\in\IN^*$) then the method is exact: $\overline{\mu}_{\infty}^{h}=\overline{\mu}_{\infty}$.
\end{theorem}

\begin{proof}
Thanks to Lemma~\ref{lem:ordre_Trace}, it is sufficient to control 
\begin{align*}
\sum_{p=1}^{+\infty}\big|\langle \bigl(Q_{\infty}-\overline{Q}_{\infty}^{h}\bigr)e_p,e_p\rangle \big|&=\sum_{p=1}^{+\infty}\frac{\sigma^2 q_p}{2(\lambda_p+b_p)}|\overline{\R}(-\lambda_p h,-b_p h)-1|\\
&\leq C\lambda_1\sigma^2 \sum_{p=1}^{+\infty}\frac{q_p|b_p|}{\lambda_p} \frac{\lambda_ph^2}{(1+\lambda_p h)^2}
\leq C\lambda_1^2\sigma^2 \sum_{p=1}^{+\infty}q_p\lambda_{p}^{-1+\s}\frac{(\lambda_p h)^{1-\s} h^{1+\s}}{(1+\lambda_p h)^2}\\
&\leq C_{\s}\lambda_1^2\sigma^2 {\rm Trace}\bigl((-A)^{-1+\s}Q\bigr) h^{1+\s},
\end{align*}
which gives the order of convergence $\s+1$ for all $\s<\overline{\s}$. Moreover, it is clear that $\overline{\R}(z,0)=1$ for any $z\leq 0$, so that if $b_p=0$ for all $p\in\IN^*$ then $\overline{Q}_{\infty}^{h}=Q_{\infty}$.
\end{proof}


\begin{remark}\label{pro:ordre_Euler}
We see in Lemma \ref{lem:aux_poly} that the error 
is zero if $z=0$ or $\beta=0$. This is related to Proposition  \ref{pro:exactLstable}  which shows that the error of the postprocessed method is zero in the linear case when $\beta=0$. This feature permits to gain one power of $h$ and thus one order of accuracy in
the proof above. 
In contrast, notice that the standard linearized implicit Euler scheme has the lower order $\s$ for all $\s<\overline{\s}$. Indeed, under the hypotheses of Theorem \ref{pro:ordre_post}, then for all $h$ small enough, $Q_{\infty}-Q_{\infty,\nu}^{h}$ is nonnegative and using Lemma~\ref{lem:ordre_Trace} we only need to control
\begin{align*}
{\rm Trace}\bigl(Q_{\infty}-Q_{\infty,\nu}^{h}\bigr)&=\frac{\sigma^2}{2}\sum_{p=1}^{+\infty}\frac{q_p}{\lambda_p+b_p}\frac{(\lambda_p+b_p)h}{2+\lambda_p h+ b_p h}\\
&=\frac{\sigma^2 h^{\s}}{2}\sum_{p=1}^{+\infty}q_p\lambda_{p}^{-1+\s}\frac{\lambda_p}{\lambda_p+b_p}\frac{\frac{\lambda_p+b_p}{\lambda_{p}^{\s}}h^{1-\s}}{2+\lambda_ph+b_ph}\\
&\leq \frac{\sigma^2 h^{\s}}{2}C_{\s}{\rm Trace}\bigl((-A)^{-1+\s}Q\bigr).
\end{align*}
Thus for all $\s<\overline{\s}$ that there exists $C_{\s}\in(0,+\infty)$ such that for all $h$ small enough
\begin{equation} \label{eq:convhsEuler}
\Big| \int_{\H}\varphi d\mu_{\infty}-\int_{\H}\varphi d\nu_{\infty}^{h} \Big| \leq C_{\s}\sup_{v\in \H}\|D^2\varphi(v)\| \sigma^2 h^{\s}.
\end{equation}
It is also possible to prove that the above Theorem~\ref{pro:ordre_post} (resp. \eqref{eq:convhsEuler}) gives optimal order of convergence, namely that \eqref{eq:convhspost} does not hold true for all $h>0$ if $\s>\overline{\s}+1$ (resp. \eqref{eq:convhsEuler} does not hold true for any $h>0$ if $\s>\overline{\s}$). This fact is also supported by the numerical simulations of Figure \ref{fig:figconv} in Section~\ref{sec:num}.
\end{remark}


\subsection{Spatial regularity analysis}
\label{sec:qualitative}

We show in this section that the postprocessed method yields a solution which has the same regularity in space as the exact solution, in contrast to the standard linearized implicit Euler method, which yields
a solution that is too smooth. The action of the postprocessing thus not only increases the order of the convergence, but also provides a qualitatively better approximation with the correct regularity.
For all Borel probability measure $\mu$ on $\H$, we define its regularity, denoted $\reg(\mu) \in \IR \cup\{-\infty,+\infty\}$, by the supremum of $s$ such that the norm $|\cdot|_s$ of $\H^{s}$ is square-integrable with respect to $\mu$:
\begin{equation}\label{eq:reg}
\reg(\mu) = \sup \{s \in \IR ,\ \int_{\H} |u|_s^2\mu(du)  < \infty\}.
\end{equation}
The interpretation in terms of random variables is the following. 
For a random variable~$v$ with values in $\H$, denoting $\mathbb{P}_v$ its probability law, we have
the identity $\IE(|v|_s^2) = \int_{\H} |u|_s^2\mathbb{P}_v(du)$ which is a finite quantity if $s<\reg(\mathbb{P}_v)$ and~$+\infty$ if $s>\reg(\mathbb{P}_v)$. %
%
Notice that instead of quantifying the regularity in terms of the Sobolev space $\H^{\s}$, one could state similar results in terms of H\"older regularity. We refer to \cite{ChW12} for such a study of the $\theta$-method applied to the stochastic heat equation with finite differences.

We focus for simplicity on the case $F=0$, $\sigma=1$, and with the initial condition $u_0=0$, but we emphasize that the extension to the general semilinear situation is straightforward, the regularity being determined only by the stochastic terms in our setting.
%
%
First, notice that for the exact solution $u(t)$, the regularity parameter is $\reg(\mathbb{P}_{u(t)}) = \reg(\mu_\infty) = \overline s$ for all $t>0$, (see also Proposition~\ref{pro:ergo_cont}). Indeed, using \eqref{eq:mild_solution}, and the  It\^o formula yields
$$\IE|u(t)|_{\s}^{2} = \sum_{p=1}^{+\infty}\frac{q_p}{2\lambda_{p}^{1-\s}}\bigl(1-\exp(-2\lambda_p t)\bigr) \begin{cases} <+\infty \quad \text{if}~\s<\overline{\s} \\ =+\infty  \quad \text{if}~\s>\overline{\s} \end{cases}.$$
The following proposition shows that at the discrete-time level, the standard linearized implicit Euler method $v_n$ in \eqref{eq:v_n} and the method without processing $u_n$ in \eqref{eq:u_n} have the regularity parameter $\overline s + 1$, whereas 
the postprocessor $\overline u_n$ in \eqref{eq:u_n_bar}  has the correct regularity parameter $\overline s$.

\begin{proposition} \label{pro:regularity}
Consider \eqref{eq:SPDE} with $F=0$, $\sigma=1$, $u_0=0$ and assume \eqref{eq:conditionTrace}.
Then, for all $h>0$ and all $n\in\mathbb{N}^*$, 
$
\reg(\mathbb{P}_{v_n}) = \reg(\mathbb{P}_{u_n}) = \reg(\nu_\infty^h) = \reg(\mu_\infty^h) = \overline s+1,$ whereas 
$\reg(\mathbb{P}_{\overline u_n}) = \reg(\overline \mu_\infty^h) = \overline s.
$
\end{proposition}
\begin{proof}
Inspecting the proof of Proposition~\ref{pro:moments}, we have $u_n=\ell_n$ and for all $\s<\overline{\s}+1, h>0,n\in\IN^*$
$$\IE\big|\ell_n\big|_{\s+1}^{2}\leq \frac{C}{h}\sum_{p=1}^{+\infty}q_p\lambda_{p}^{\s}\frac{(\lambda_p h) ^2}{(1+\lambda_p h)^2-1} \leq   \frac{C}{h}{\rm Trace}\bigl((-A)^{\s-1}Q\bigr);$$
this yields $\reg(\mathbb{P}_{u_n}) = \reg(\mu_\infty^h) \geq \overline s+1$. The reverse inequality is obtained with a similar lower bound. 
The proof for the standard linearized implicit Euler scheme $v_n$ is similar.
%
Now, adding the postprocessing,
$\overline{u}_{n}=u_n+\frac{1}{2}J_3\sigma \sqrt{h}\xi_{n}^{Q},$ we obtain using the definition of $J_3$,
\begin{align*}
\IE\big|J_3\sigma \sqrt{h}\xi_{n}^{Q}\big|_{\s}^{2}&=h{\rm Trace}\bigl((I-\frac{h}{2}A)^{-1}(-A)^{\s}Q\bigr)
=\sum_{p=1}^{+\infty}\frac{q_p}{\lambda_{p}^{1-\s}}\frac{\lambda_ph}{1+\frac{1}{2}\lambda_ph}\begin{cases} <+\infty \quad \text{if}~\s<\overline{\s}, \\ =+\infty  \quad \text{if}~\s>\overline{\s}. \end{cases}
\end{align*}
The term $\frac{1}{2}J_3\sigma \sqrt{h}\xi_{n}^{Q}$ thus has exactly the same regularity $\overline{\s}$ as the exact solution.
This concludes the proof of $\reg(\mathbb{P}_{\overline u_n}) = \reg(\overline \mu_\infty^h) = \overline s$.
\end{proof}

\section{Numerical experiments}
\label{sec:num}

\begin{figure}[b!]
\centering
\medskip
\scalebox{0.80}{\global\def\path{#1}\input{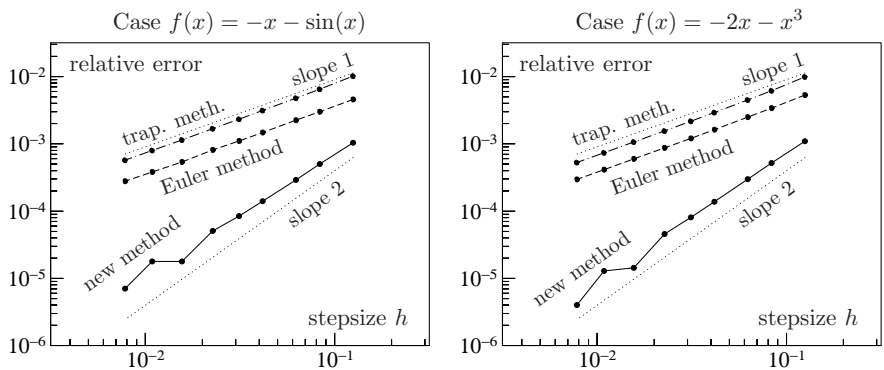}}
\caption{
Comparison of the new method (solid lines) with 
the standard linearized implicit Euler method (dashed lines) and the trapezoidal method (dashed-dotted lines) for the scalar SDE~\eqref{eq:heatspdedisc} (dimension $N=1$) with $A=-1,\sigma=1$, 
nonlinearity $f(x)$. 
Error for $\mathbb{E}(\exp(-X(T)^2))$ at final time $T=1$ versus the stepsize $h$, 
where $1/h=8,12,16,24,32,44,64,92,128$. Averages over $10^{10}$ samples. 
\label{fig:figconvfinie}}
\end{figure}

In this section, we compare numerically the performances of the new postprocessed method \eqref{eq:newmeth}
with the standard linearized implicit Euler method \eqref{eq:v_n}, and the trapezoidal method
\eqref{eq:CN}, both in finite and infinite dimensions.

We consider first in Figure \ref{fig:figconvfinie} the scalar nonlinear SDE
\eqref{eq:heatspdedisc} with dimension $N=1$, parameters $A=-1,\sigma=1$ and the initial condition $X(0)=0$. 
Taking $f_1(x)=Ax$ and $f_2(x)=f(x)$ in \eqref{eq:newmeth}, we consider the nonlinearities $f(x)=-x-\sin(x)$ and $f(x)=-2x-x^3$, respectively, and we compute the averages over $10^{10}$ independent trajectories with final time $T=1$ and compare for many time stepsizes the accuracy
for $\mathbb{E}(\exp(-X(T)^2)) = \int_{-\infty}^{+\infty} \exp({-x^2}) \rho(x)dx$.
The final time $T=1$ is chosen large enough so that the equilibrium is reached and the exponentially decaying term 
$e^{-\lambda T}$ in \eqref{eq:th2} is negligible.
In the left picture of Figure \ref{fig:figconvfinie} where the nonlinearity $f(x)=-x-\sin(x)$ is Lipschitz, we observe as shown in Theorem \ref{thm:order2} the expected order $2$ of convergence for the new method, while the standard methods exhibit order $1$ of convergence (see the reference lines with slopes $1,2$). Although our analysis in  Section 
\ref{sec:nonlineardimfinie} applies only to globally-Lipschitz vector fields, we observe that the excellent performances 
of the new method persist also in the example with the non-Lipschitz nonlinearity $f(x)=-2x-x^3$ and the globally bounded test function $\phi(x)=\exp(-x^2)$ (right picture of Figure \ref{fig:figconvfinie}). 

\begin{figure}[tb]
\small
\centering
\subfigure[Standard linearized Euler method, sample trajectory $u(x,t)$ and corresponding profile at final time $t=1$.]{
\scalebox{0.85}{\includegraphics[width=0.5\textwidth]{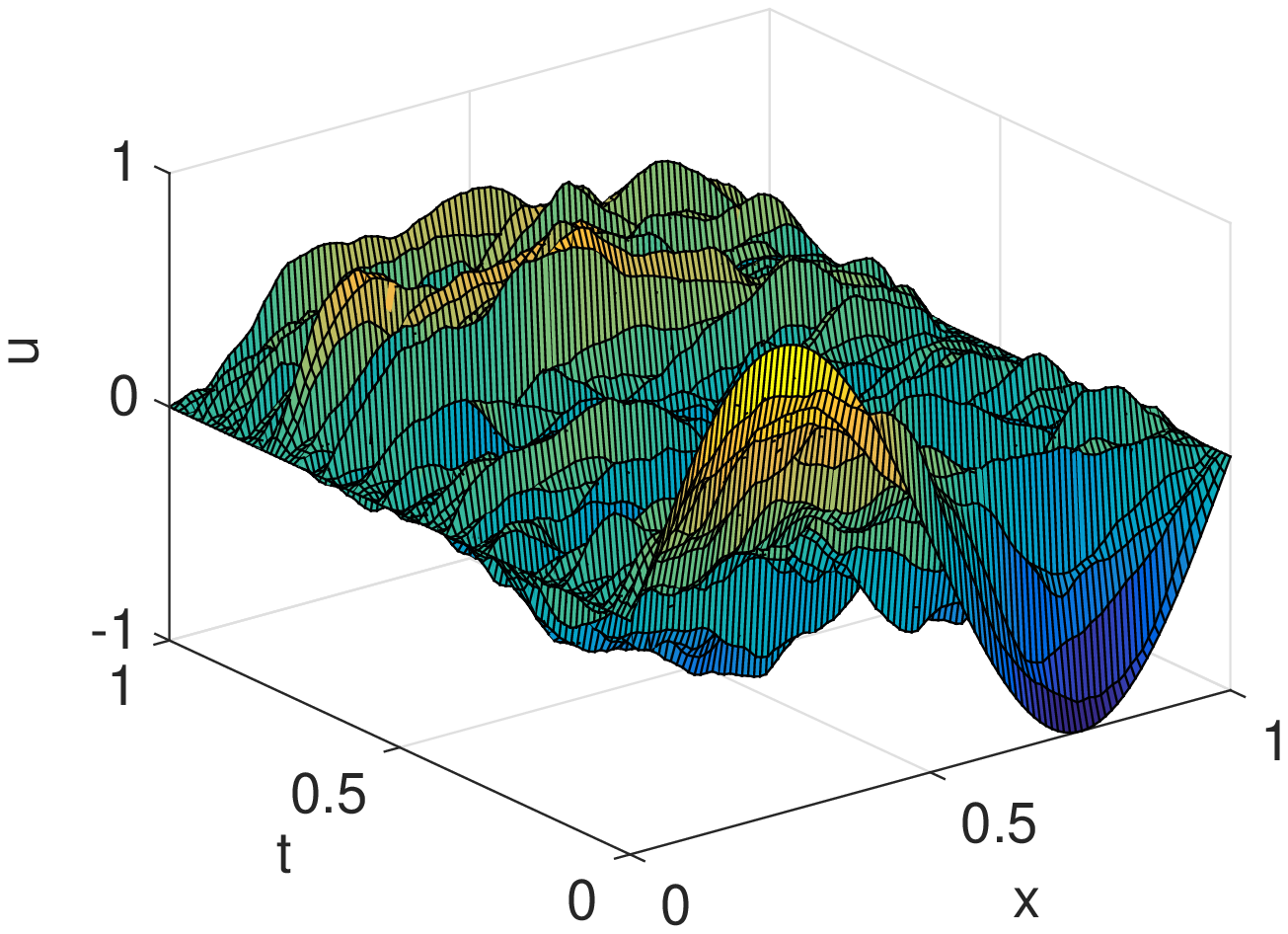} \qquad
\includegraphics[width=0.47\textwidth]{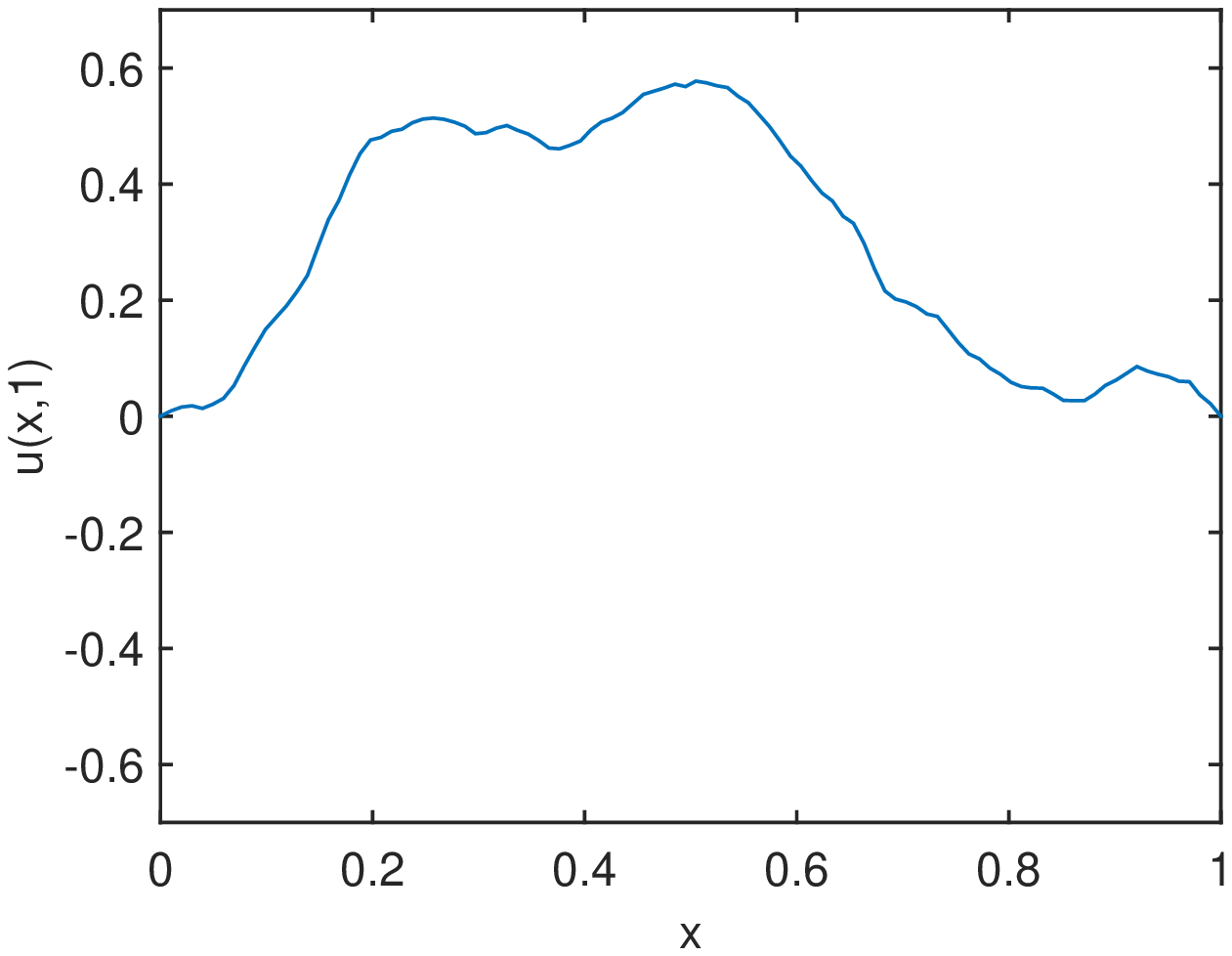}}
}
\subfigure[New method with postprocessor, sample trajectory $u(x,t)$ and corresponding profile at final time $t=1$.]{
\scalebox{0.85}{\includegraphics[width=0.5\textwidth]{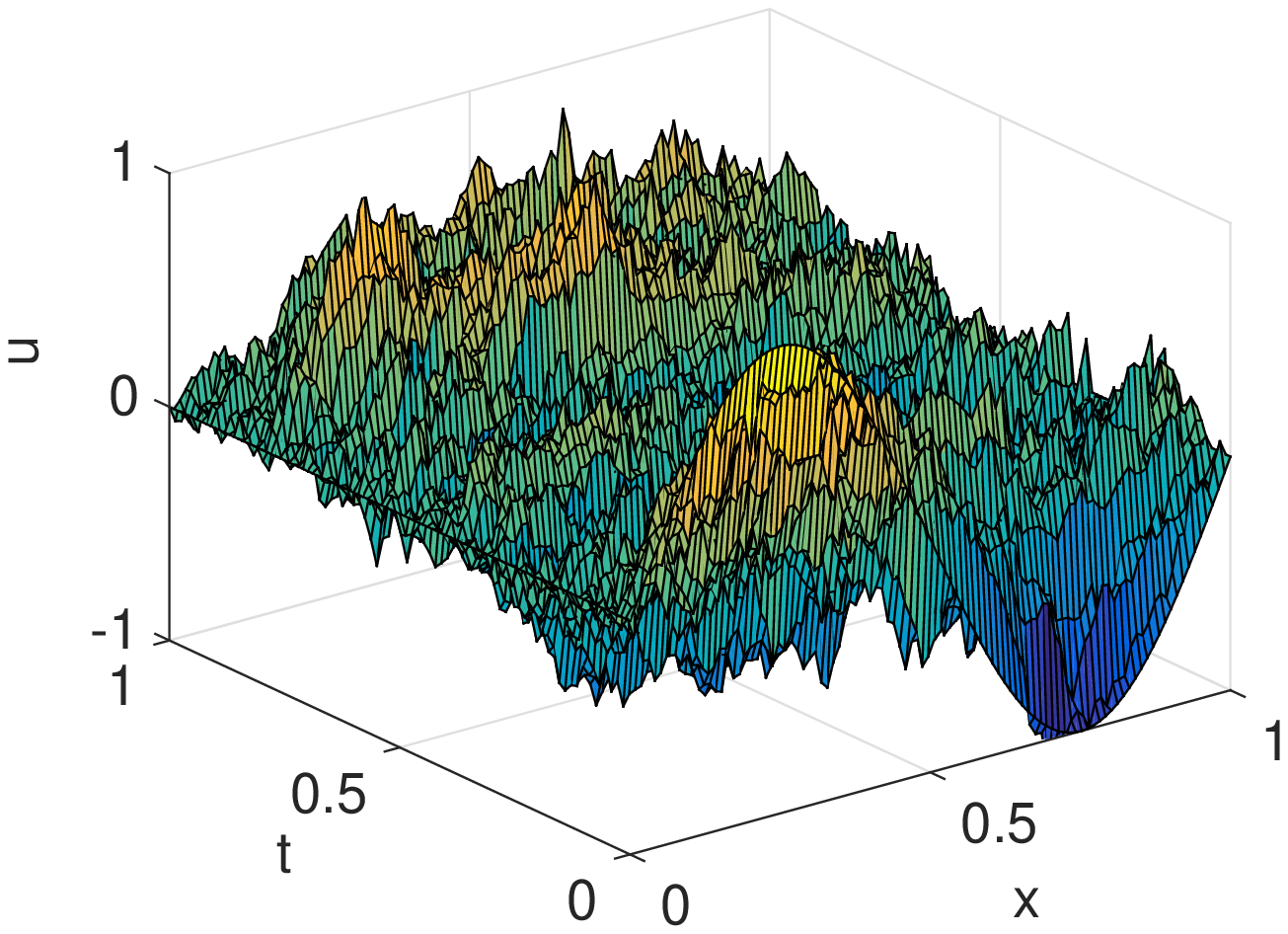} \qquad
\includegraphics[width=0.47\textwidth]{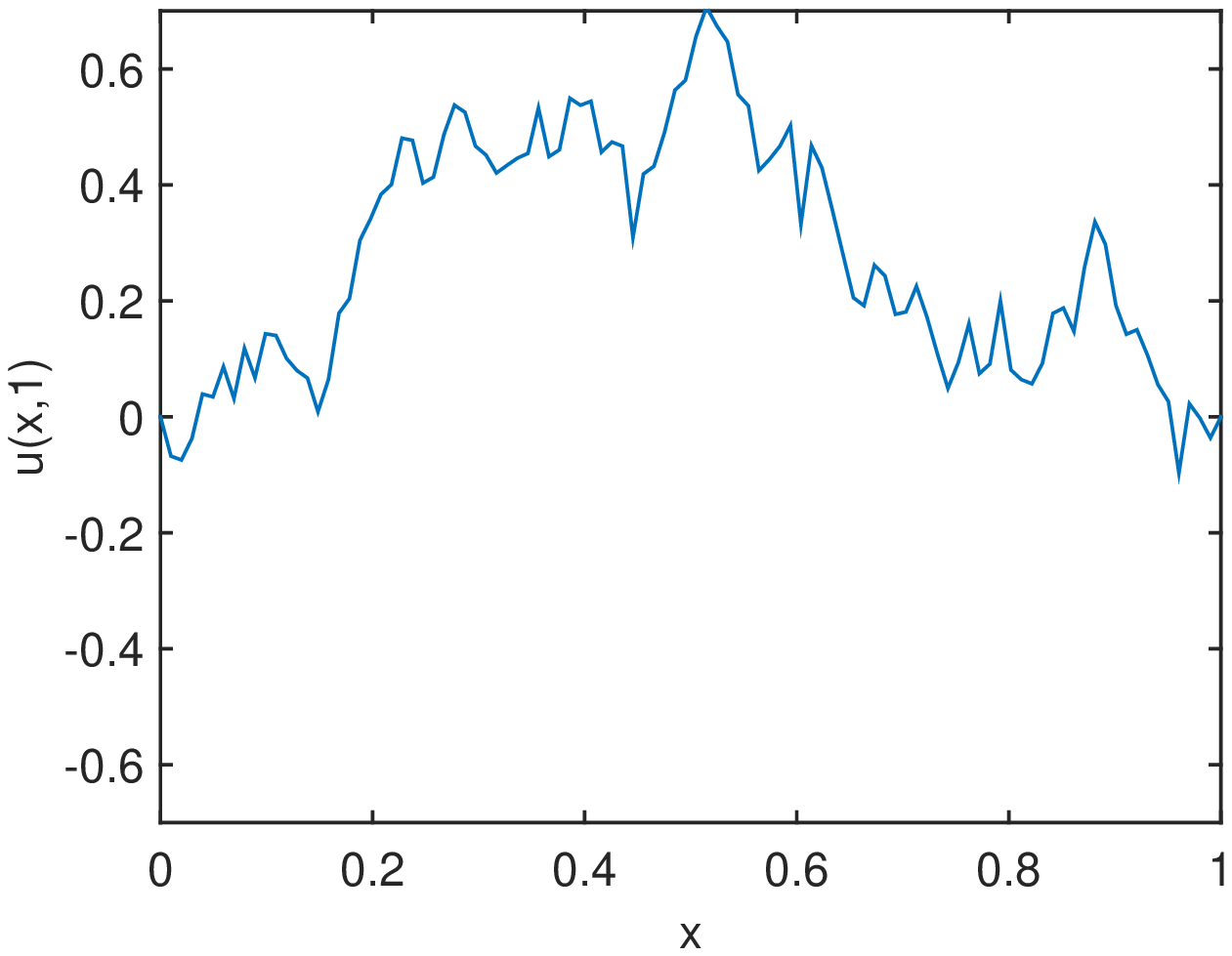}}
}
\caption{
Samples of realisation of the stochastic nonlinear heat equation \eqref{eq:heat1d} with $f(u)=-u-\sin(u)$ using the standard linearized implicit Euler method and the new method with postprocessor. With $N=100$ space grid points and timestep size $h=1/100$.
\label{fig:figtraj}}
\end{figure}
We next consider a standard finite-difference approximation $U_j(t) \simeq u(j\Delta x,t)$ of the 1D heat equation \eqref{eq:heat1d} with zero Dirichlet boundary conditions on a uniform grid with size $\Delta x = 1/(N+1)$. This yields
the following system of SDEs in dimension $N$,
\begin{equation*} 
dX(t) = \frac1{\Delta x^2}\begin{pmatrix} 
-2 & 1 \\
1 & -2 & 1 \\
& \ddots &\ddots &\ddots  \\
&& 1 & -2
\end{pmatrix}
\begin{pmatrix} X^1(t) \\ X^2(t) \\\vdots \\X^N(t) \end{pmatrix} dt 
+ \begin{pmatrix} f(X^1(t)) \\ f(X^2(t)) \\\vdots \\f(X^N(t)) \end{pmatrix} dt
+
\frac1{\sqrt{\Delta x}} 
\begin{pmatrix} dW^1(t) \\ dW^2(t) \\\vdots \\ dW^N(t) \end{pmatrix},
\end{equation*}
where $W^1,\ldots,W^N$ are independent one-dimensional standard Wiener processes.

Considering $N=100$ grid points for the space discretization, we take the initial condition
$u(x,0)=\sin(2\pi x)$ and plot in Figure \ref{fig:figtraj} a sample trajectory on the time interval $(0,1)$ for the 
standard linearized implicit Euler method \eqref{eq:v_n} and the new method with postprocessor $\overline X_n$, using the same sets of generated random numbers for both methods. 
We observe that the solution of the standard linearized implicit Euler method (Fig.\ts\ref{fig:figtraj}(a)) is qualitatively too smooth compared to the new method (Fig.\ts\ref{fig:figtraj}(b)) with 
the postprocessor applied at each timestep, which corroborates the statement of Proposition \ref{pro:regularity} in Section \ref{sec:qualitative}.  Notice that the trajectory for the new method without applying the postprocessor would look very similar to that of the standard linearized implicit Euler method. The spatial regularity observed in Figure~\ref{fig:figtraj}(b) is the same as the one of a diffusion process driven by Brownian motion, conditioned to be zero at initial and final times. This property is not surprising, see \cite{HSVW05,HSV07} for a study of the link between the law of a conditioned diffusion and invariant distributions of SPDEs.

\begin{figure}[htb]
\centering
\medskip
\scalebox{0.80}{\global\def\path{#1}\input{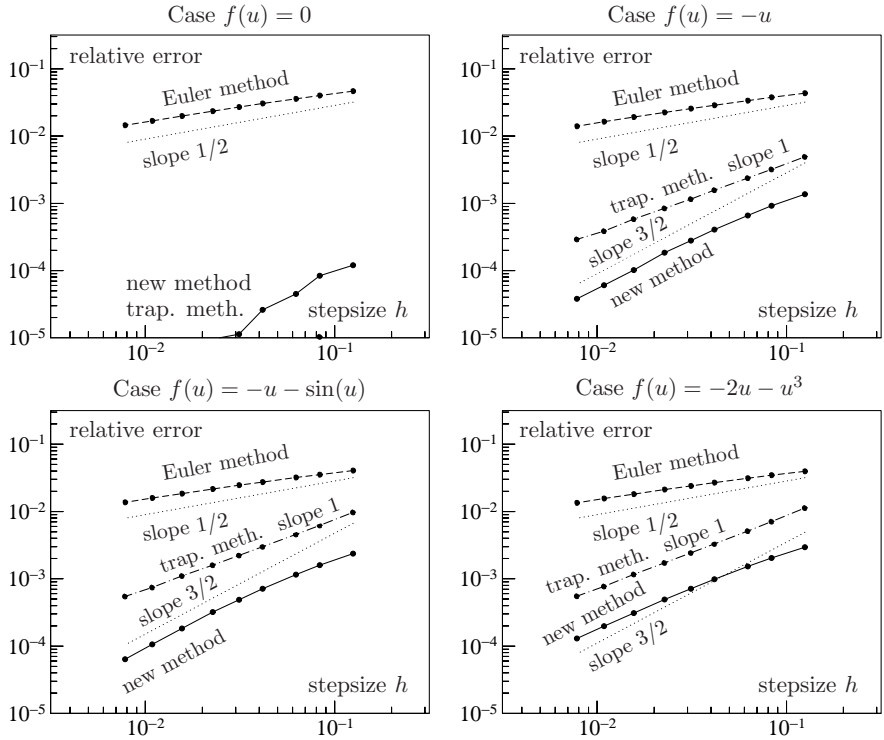}}
\caption{
Comparison of the new method (solid lines) with 
the standard linearized implicit Euler method (dashed lines) and the trapezoidal method (dashed-dotted lines)
for the stochastic heat equation \eqref{eq:heat1d} with nonlinearity $f(u)$ discretized in space with $N=100$ grid points. 
Error for $\mathbb{E}(\exp(-\|u(T)\|^2))$ at final time $T=1$ versus the stepsize $h$, 
where $1/h=8,12,16,24,32,44,64,92,128$. Averages over $10^9$ samples. 
\label{fig:figconv}}
\end{figure}

We finally plot in Figure \ref{fig:figconv}
the error at final time $T=1$ for the quantity $\mathbb{E}(\exp(-\|u\|_{L^2(0,1)}^2))$ where we use the approximation
$
\|u\|_{L^2(0,1)}^2 \simeq \Delta x\sum_{j=1}^N (X^j)^2.
$
We use arbitrarily the initial condition $u(x,0)=0$ and we compute the averages over $10^9$ independent trajectories
so that the Monte-Carlo errors become negligible. The reference solution is computed using the new method with stepsize $h=1/512$.
We consider respectively, the cases of various nonlinearities $f(u)=0$, $f(u)=-u$, $f(u)=-u-\sin(u)$, $f(u)=-2u-u^3$.
We observe in all cases an order of convergence $1/2$ for the standard linearized implicit Euler method, while the
trapezoidal method has order $1$. 
In contrast to the standard linearized implicit Euler method, the new method with postprocessor has a much better accuracy by a factor $15-250$ in the range of stepsizes considered. 
It has a zero bias for $f(u)=0$ and order $3/2$ in the linear case $f(u)=-u$, as proved in Theorem \ref{pro:ordre_post}.
We observe that the order of convergence persists in the nonlinear case $f(u)=-u-\sin(u)$, and the excellent accuracy
persists in the non-Lipschitz case $f(u)=-2u-u^3$, although an order reduction can be observed.


\vskip0.4mm

\noindent \textbf{Acknowledgements.} 
The research of C.-E.\,B. and G.\,V. is partially supported by the Swiss
National Science Foundation, Grant: No\,200020-149871/1 and No\,200020\_144313/1, respectively. 
The computations were performed at University of Geneva on the Baobab cluster.


\section*{Appendix}

\begin{proof}[Proof of Lemma \ref{lemma:order2}]
We consider the variant \eqref{eq:abc} of the linearized implicit Euler method. 
A straightforward calculation yields the following weak expansion for all $\phi\in\CC_P^\infty(\IR^N,\IR)$,
\begin{equation}\label{eq:expansion1}
\IE(\phi(Y_1)|Y_0=x) = \phi(x) + h\mathcal{L}\phi(x) + h^2\mathcal{A}_1 \phi(x) + \bigo(h^3),
\end{equation}
where the constant symbolized by $\bigo$ is independent of $h$ but depends on $\phi$ and depends on $x$ with a polynomial 
growth.
Here, the fourth-order linear differential operator $\mathcal{A}_1$ is given by\footnote{
We denote $\phi'(x):\mathbb{R}^N \rightarrow \IR$ the first derivative of $\phi$ at point $x\in\mathbb{R}^N$, $\phi''(x)$ the second derivative (a symmetric bilinear form on $\IR^N\times \IR^N$), $\phi'''(x)$ the third derivative (a symmetric bilinear form), etc.
}
\begin{eqnarray*}
\mathcal{A}_1\phi &=& \frac12 \phi''(f_0,f_0) + \frac{\sigma^2}2 \sum_{i=1}^N \phi'''(e_i,e_i, f_0)
+ \frac{\sigma^4}8  \sum_{i,j=1}^N \phi^{(4)} (e_i,e_i,e_j,e_j) + \phi' f_1'f_0\nonumber\\
&+&  \Big((a_1+1)^2\Big) \frac{\sigma^2}2\phi' \sum_{i=1}^N f_1''(e_i,e_i)
+ {\Big(a_1 +1 + a_3\Big) } \sigma^2\sum_{i=1}^N \phi''(f_1'e_i,e_i) \\
&+&  a_2^2 \frac{\sigma^2}2\phi' \sum_{i=1}^N f_2''(e_i,e_i)
+ {a_2  } \sigma^2\sum_{i=1}^N \phi''(f_2'e_i,e_i)
\end{eqnarray*}
where $e_1,\ldots,e_N$ denotes the canonical basis of $\mathbb{R}^N$.

For any $\psi\in\CC_P^\infty(\IR^N,\IR)$, we denote $\langle \psi\rangle=\int_{\IR^N}\psi(y)\mu_\infty(dy)=\int_{\IR^N}\psi(y)\rho(y)dy$. By integration by parts, and using the identity $\nabla \rho = \frac{2}{\sigma^2}\rho f_0$, we obtain,
\begin{eqnarray} 
\left\langle \phi''(f_0,f_0) \right\rangle &=&   \textstyle
\left\langle   -\phi'(f_0'f_0 + (\mathrm{div}\, f_0) f_0 + \frac2{\sigma^2} \|f_0\|^2 f_0) \right\rangle, 
\nonumber\\
\textstyle\left\langle \sigma^2 \sum_i \phi'''(f_0,e_i,e_i) \right\rangle    &=&  \textstyle
\left\langle  \phi'(\sigma^2  \sum_i f_0''(e_i,e_i) + 4f_0'f_0 + 2 (\mathrm{div}\,f_0) f_0 + \frac4{\sigma^2} \|f_0\|^2 f_0) \right\rangle,\nonumber\\
\textstyle \left\langle \sigma^2\sum_{ij}\phi^{(4)}(e_i,e_i,e_j,e_j) \right\rangle &=&  \textstyle
\left\langle  -\sum_i 2\phi'''(f_0,e_i,e_i) \right\rangle,\nonumber\\
\textstyle \left\langle   \sigma^2 \sum_i \phi''(f_1'e_i,e_i) \right\rangle &=&  \textstyle
\left\langle  -\phi'(\sigma^2 \sum_i f_1''(e_i,e_i) + 2f_1' f_0) \right\rangle,
\label{eq:estimparts}
\end{eqnarray}
see \cite{AVZ14a,Vil15} for examples of such calculations. Then
\begin{eqnarray*}
\left\langle \mathcal{A}_1\phi\right\rangle
&=& \left\langle\Big(a_1^2+2 a_1\Big) \frac{\sigma^2}2\phi' \sum_{i=1}^N f_1''(e_i,e_i)
+ {\Big(\frac14 + a_1 + a_3\Big) } \sigma^2\sum_{i=1}^N \phi''(f_1'e_i,e_i) \right. \\
&+& \left. a_2^2 \frac{\sigma^2}2\phi' \sum_{i=1}^N f_2''(e_i,e_i)
+ {\Big(-\frac14 + a_2 \Big) } \sigma^2\sum_{i=1}^N \phi''(f_2'e_i,e_i) \right\rangle
\end{eqnarray*}
We note that the postprocessor $Y_n \mapsto\overline Y_n$ in \eqref{eq:abc} satisfies 
\eqref{eq:phiG} with $q=1$ and
$\mathcal{\overline A}_1 \phi= b_1\phi' f_1 + b_2\phi' f_2  + \frac{c^2}2 \sigma^2 \Delta \phi$. 
\begin{equation*}
[\mathcal{L},\mathcal{\overline A}_1]\phi = 
(b_1-b_2)  \phi'(f_2'f_1-f_1'f_2) 
-\frac{c^2}2 \sigma^2\phi' \sum_{i=1}^N f_0''(e_i,e_i)
 {-c^2 \sigma^2} \sum_{i=1}^N \phi''(f_0'e_i,e_i). 
\end{equation*}
We deduce the following expression for $\left\langle \mathcal{A}_1\phi + [\mathcal{L},\mathcal{\overline A}_1]\phi\right\rangle$, which is the
key quantity to compute in Theorem \ref{thm:main},
Summing up, we obtain
\begin{eqnarray*} 
&&\left\langle \mathcal{A}_1\phi + [\mathcal{L},\mathcal{\overline A}_1]\phi\right\rangle\\
&=&\Big\langle\Big(a_1^2+2 a_1 + b_1-c^2\Big) \frac{\sigma^2}2\phi' \sum_{i=1}^N f_1''(e_i,e_i)
+ {\Big(a_1 + a_3 + \frac14 + b_1-c^2\Big) } \sigma^2\sum_{i=1}^N \phi''(f_1'e_i,e_i)  \nonumber\\
&+&\Big(a_2^2 + b_2-c^2\Big) \frac{\sigma^2}2\phi' \sum_{i=1}^N f_2''(e_i,e_i)
+ {\Big(-\frac14 + a_2  + b_2-c^2\Big) } \sigma^2\sum_{i=1}^N \phi''(f_2'e_i,e_i)
\nonumber\\
&+&(b_1-b_2) \phi'(f_2'f_1-f_1'f_2)
\Big\rangle. 
\end{eqnarray*}
The above quantity is zero by assumption for all test function $\phi$.
This means $(\mathcal{A}_{1} + [\mathcal{L},\mathcal{\overline A}_1])^*\rho = 0$.
Applying Theorem \ref{thm:main}  with $q=1$  to the scheme \eqref{eq:abc} then yields for $\overline Y_n$ in \eqref{eq:abc} a method of weak order two for the invariant measure. This concludes the proof of Lemma \ref{lemma:order2}.
\end{proof}


\bibliographystyle{abbrv}
\bibliography{paper_highorder_spde}

\end{document}